\documentclass[leqno,11pt,a4paper,twoside]{article}%
\usepackage{amssymb}
\usepackage{amsfonts}
\usepackage{amsmath}
\usepackage{authblk}
\usepackage{cancel}
\usepackage{comment}
\usepackage{graphicx}
\usepackage{xcolor}
\usepackage{enumitem}
\usepackage[normalem]{ulem}
\usepackage{lipsum,multicol}
\usepackage[english]{babel}
\setlist[itemize]{noitemsep, topsep=0pt}
\setlist[enumerate]{noitemsep, topsep=0pt}
\setlist[itemize]{leftmargin=*}
\setlist[enumerate]{leftmargin=*}
\providecommand{\U}[1]{\protect\rule{.1in}{.1in}}
\makeatletter

\newcommand{\Rmnum}[1]{\expandafter\@slowromancap\romannumeral #1@}
\makeatother

\providecommand{\norm}[1]{\left\lVert#1\right\rVert}
\providecommand{\abs}[1]{\left\lvert#1\right\rvert}
\providecommand{\pr}[1]{\left(#1\right)} 
\providecommand{\pp}[1]{\left[#1\right]} 
\providecommand{\set}[1]{\left\lbrace#1\right\rbrace} 
\providecommand{\scal}[1]{\left\langle#1\right\rangle}

\newtheorem{theorem}{Theorem}[section]

\newtheorem{assumption}{Assumption}[section]

\newtheorem{corollary}{Corollary}[section]

\newtheorem{definition}{Definition}[section]

\newtheorem{lemma}{Lemma}[section]

\newtheorem{proposition}{Proposition}[section]
\newtheorem{remark}{Remark}[section]
\newtheorem{example}{Example}[section]

\newenvironment{proof}[1][Proof]{\noindent\textbf{#1.} }{\ \rule{0.5em}{0.5em}}

\usepackage{hyperref}
\usepackage{cite}

\begin{document}

\title{Linearisation Techniques and the Dual Algorithm for a Class of Mixed Singular/Continuous Control Problems in Reinsurance. \\ Part \Rmnum{1}: Theoretical Aspects\footnote{The work has been supported by National Key R and D Program of China (NO. 2018YFA0703900), the NSF of P.R.China (NOs. 12031009, 11871037), NSFC-RS (No. 11661130148; NA150344), 111 Project (No. B12023).}}

\author[1,2]{Dan GOREAC}
 \author[1,3]{Juan LI}
 \author[1]{Boxiang XU}
\affil[1]{School of Mathematics and Statistics, Shandong University, Weihai, Weihai 264209, P.R. China}
\affil[2]{LAMA, Univ Gustave Eiffel, UPEM, Univ Paris Est Creteil, CNRS, F-77447 Marne-la-Vall\'{e}e, France}
\affil[3]{Research Center for Mathematics and Interdisciplinary Sciences, Shandong University, Qingdao 266237, P. R. China.}
\affil[ ]{{\it E-mails: dan.goreac@univ-eiffel.fr,\,\ juanli@sdu.edu.cn,\,\ boxiangxu@163.com}}
\date{}
\maketitle

\noindent \textbf{Abstract}.
This paper focuses on linearisation techniques for a class of mixed singular/continuous control problems and ensuing algorithms. The motivation comes from (re)insurance problems with reserve-dependent premiums with Cram\'{e}r-Lundberg claims, by allowing singular dividend payments and capital injections.  Using variational techniques and embedding the trajectories in an appropriate family of occupation measures, we provide the linearisation of such problems in which the continuous control is given by reinsurance policies and the singular one by dividends and capital injections.  The linearisation translates into a dual dynamic programming (DDP) algorithm. An important part of the paper is dedicated to structural considerations allowing reasonable implementation.  We also hint connections to methods relying on moment sum of squares and LMI (linear matrix inequality)-relaxations to approximate the optimal candidates.
\bigskip


\bigskip


\section{Preliminaries}
\subsection{Introduction}

\quad\ \ This paper focuses on an insurance model with Cram\'{e}r-Lundberg claims ($C$ arriving along a Poisson process $N$) and mixed types of control policies $\pi$. These policies include continuous retention levels $u$ affecting the reserve-dependent premium $p$ and can be seen as a reinsurance mechanism in which the primary company only pays the fraction $u(C)$ of one claim.  The policies also include singular controls coming from classical dividend payments ($L$) and capital injections ($I$).
\begin{equation*}
X_t^\pi:=x+\int_0^t p^{u_s}\pr{X_s^\pi}ds-\sum_{i=1}^{N_t}u_{\tau_i}\pr{C_i}-L_t+I_t.
\end{equation*}
The capital injected has to be reimbursed at a unitary cost $k>1$. The aim is to maximize a $q>0$-discounted value of dividends from which the reimbursement is subtracted.
\begin{equation*}
\textnormal{Maximize over }\pi,\
\mathbb{E}_x\pp{\int_0^{\sigma_{0-}^{\pi}}e^{-qs}\pr{dL_s-kdI_s}}.
\end{equation*}The precise restrictions on the policies $\pi$ are made clear below.
The literature  on insurance with various models is very rich and a complete overview exceeds the aim of this paper. \\
The first result in connection with dividends optimization is due to de Finetti \cite{deFinetti} for Brownian approximations of the claims in which injections of capital are not possible and the bankruptcy is declared as soon as the reserve becomes negative.  Within the same Brownian setting,  \cite{shreve1984optimal} considers a problem in which the reserve is systematically reflected to the non-negative real values and the injections (deposits) are withdrawn from the efficiency criterion.  The paper \cite{lokka2008optimal} is the first to compare the two values and to specify a critical cost of injection of capital discriminating between the two strategies. \\
In the jump-diffusion case, a complete characterization of optimal dividends has been provided in \cite{avram2007optimal} for a class of spectrally negative Lévy claims in absence of injections of capital. Extensions to systematic injections (similar to \cite{shreve1984optimal}) have made the object of \cite{kulenko2008optimal}, \cite{eisenberg2011minimising}, \cite{noba2020bailout} and so on. The paper\cite{avram2019lokka} focuses on a control problem in connection to an insurance company that pays out dividends and is allowed to inject capital and shows that the a L{\o}kka-Zervos-type alternative is also valid for a Cram\'{e}r-Lundberg risk process in which the claims are exponentially distributed.  The paper \cite{avram2020equity} deals with similar problems but without asking a systematic reflection (as it was the case in the previously cited papers), but optimizing the level below which such injections are futile and declaring bankruptcy is preferable. The same kind of results have been obtained in \cite{gajek2017complete} for more general claims but using different methods based on the scale functions (see \cite{asmussen2010ruin}, \cite{kyprianou2014fluctuations},  \cite{Liu_2014}etc.).  These papers usually deal with a fixed premium $p$ and they do not take into account the question of reinsurance. The optimal strategy in \cite{avram2020equity} (and \cite{gajek2017complete}) is shown to be of $a,b$-type (reflect if above $-a$, pay dividends if above $b$), with $a$ and $b$ depending on $k$.
\\
Finally, for reinsurance problems (but without injections), the reader is referred to the book \cite{azcue2014stochastic} and the seminal paper \cite{azcue2005optimal} to get acquainted with the Hamilton-Jacobi approach. For excess-of-loss reinsurance, also see \cite{Albrecher_2011}.\\

We consider here a model that extends the one in \cite{avram2020equity} by adding (to the premium as it is already the case in \cite{azcue2014stochastic}) a continuous control parameter specifying the level of retention in connection to reinsurance.  Verification results for such models are a lot more trickier.
The main methods we wish to suggest in this article in order to solve mixed singular/continuous  optimal control problems are based on linear programming techniques and dual dynamic programming algorithms. Over the years,  similar methods have been employed for control problems (but lacking the mixed features). Using viscosity arguments, it is shown that the initial control problem is equivalent to a linear optimization problem formulated on a space of measures ( see \cite{buckdahn2011stochastic}, \cite{goreac2011mayer} and references therein for more details for diffusions). The same approach is adapted to continuous control of jump diffusions, e.g.  \cite{Serrano2015OnTL}. For example, the paper \cite{taksar1997infinite} provides an equivalent infinite-dimensional linear programming and a dual formulation for multidimensional singular stochastic control.\\

As for the dual algorithm, Pereira et al.  are the first to have presented a methodology for the solution of multi-stage optimization problems with random features called stochastic dual dynamic programming (SDDP) in  \cite{pereira1991multi}. Several scholars devoted themselves to the improvement of DDP. For example, \cite{shapiro2011analysis} presents statistical properties as well as the convergence of SDDP methods applied to multi-stage linear programming problems. Lasserre et al.  consider  nonlinear optimal control problems treated through occupation measures with polynomial data and provides a method named LMI (linear matrix inequality)-relaxations in \cite{lasserre2008nonlinear} and \cite{lasserre2005nonlinear}. Recently, \cite{hohmann2020moment} present a finite-horizon optimization algorithm based on DDP by using sum-of-squares techniques and LMI-relaxations to solve the problems described by occupation measures with polynomial functions. These references mostly concern deterministic dynamics, continuous controls and they do not present stopping times. In connection with exit problems with stochastic dynamics, we mention \cite{henrion2021moment}, but as far as we see, their problem has no mixed features for control, nor does it present representations for stopping times. \\

\textbf{Main contributions and positioning.} The main difference with respect to our work is that the aforementioned papers deal with a specific type of control (either continuous or singular) and, foremost, as we have already mentioned, they do not take into account linear formulations for the stopping times. These features are essential to our problem and need to be dealt with.
\begin{enumerate}
\item Rather than relying on the classical duality for infinite-dimensional programs (as it is the case in \cite{taksar1997infinite}), we adopt a variational point of view inspired by \cite{buckdahn2011stochastic}.  We prove, in Proposition \ref{Vsupersol} that the value function of our control problem \eqref{Value} satisfies a Hamilton-Jacobi variational inequality \eqref{HJB}. This super-solution can be approximated by regular functions (cf. Proposition \ref{PropSupersol}).
\item We linearise the mixed singular/continuous optimal control problem by defining occupation measures satisfying convenient linear and total-variation restrictions in \eqref{StrProperties}. To our best knowledge, this kind of formulations are completely new.
\item We provide two kinds of dual formulations for the control problem in Theorem \ref{ThmDual}. This translates into linearised forms of the dynamic programming principle (cf.  \eqref{TwoStage}), again new as far as we know.
\item Since the Hamiltonian is written differently on the negative axis\footnote{the reserve can become negative without inducing bankruptcy because of capital injections}, the set of constraints, and the dual formulation have to be extended in a compatible way (see Proposition \ref{PropNegx}). Again, this is specific to our framework and we have no knowledge of similar results.
\item Finally, we present the two-stage dual algorithm inspired by the cited references. However, a rather important part of the paper is spent explaining
\begin{itemize}
\item how to deal with the non-compactness of the features (e.g. possibly unbounded injections);
\item how to ensure polynomial features for the infinitesimal generator (such that the considerations in \cite{putinar1993positive} apply);
\item how to generate scenarios in the forward step.
\end{itemize}
Again, to our best knowledge, such features are new.
\end{enumerate}
Let us just point out that the numerical illustrations of the algorithm make the object of a second part in order to keep the paper at a reasonable length. They include comparisons with the $a,b$-policies benchmark in \cite{avram2020equity} and \cite{AGAS2022}.\\

\textbf{This paper is organized as follows. } In the remaining of the section, we give some standard notations used throughout the paper. Section \ref{Section2} is devoted to the description of the problem. In particular, the surplus process is introduced in \eqref{Surpluspi} and the restrictions on the dividends are specified in \eqref{RestrDividends}. Having defined the value function of our control problem in \eqref{Value}, we provide the upper estimates for the controlled surplus trajectory and the regularity of the value function in Proposition \ref{PropBasicPropertiesXV}. Further estimates used to define variation constraints for the occupation measures are exhibited in Proposition \ref{PropEstimILX}. Section \ref{Section3} is concerned with the variational characterization of the value function. This is done in Proposition \ref{Vsupersol} and completed with approximation results in Proposition \ref{PropSupersol}. The linearization method makes the object of Section \ref{Section4}. We describe the occupation measures and their constraints in \eqref{StrProperties} and Definition \ref{Theta}. The duality and no-gap results make the object of Theorem \ref{ThmDual}. Section \ref{Section5} provides further insight into the linearized dynamic programming principles in Subsection \ref{Sub5.1} and the extensions to the negative axis in Subsection \ref{Sub5.2}. These considerations are followed by a detailed description of the algorithm. Besides the presentation of the forward and backward step, we make a thorough analysis of the compact-reduction of the state and control spaces, and of the polynomial compatibility of the generator, see Example 5.1. Hints to the consideration of a bankruptcy penalty are provided in Section \ref{Section6}. Finally, the Appendix gathers all the proofs for the theoretical results.

\subsection{Notations}
We use the following standard notations:\begin{itemize}
\item $\mathbb{R}$ stands for the real axis, $\mathbb{R}_+$ denotes the non-negative real semi-axis and $\mathbb{R}_-$ denotes the non-positive real semi-axis;
\item $\min$ (resp., $\wedge$) denotes the minimum between several (resp., two) real quantities; similarly, $\max$ (resp., $\vee$) denotes the maximum. For a real quantity $x$, we set $x^+:=x\vee 0$.
\item Bounded variation functions are often identified with the associated measures and the Stieltjes integral is indifferently written with respect to a function or the associated measure.
\item Given a metric space $E$,
\begin{itemize}
\item $\mathcal{B}(E)$ denotes the family of Borel subsets of $E$;
\item $\mathcal{P}(E)$ denotes the family of probability measures on $E$;
\item $\mathcal{M}^+(E)$ stands for the family of positive Borel measures on $E$.
\item For a positive real constant $a>0$, $\mathcal{M}^+_a(E)$ stands for the the family of measures in $\mathcal{M}^+(E)$ whose total mass does not exceed $a$.
\item The space $C^1=C^1\pr{\mathbb{R}}$ stands for differentiable real functions with continuous derivative; the space $C^1\pr{\mathbb{R};\mathbb{R}_+}$ stands for non-negative functions in $C^1$; the space $C^1_{lin}$ (resp. $C^1_{lin}\pr{\mathbb{R};\mathbb{R}_+}$) stands for the respective subsets of functions with (at most) linear growth.
\item We let $\mathbb{R}_r[x]$ stand for polynomials of degree not exceeding $r\in\mathbb{N}^*$ in the Euclidean variable $x\in\mathbb{R}^N$,  and we designate by $\text{deg}(p)$ the degree of such $p\in\mathbb{R}_r [x]$.
\end{itemize}
\end{itemize}

\section{Description of the Problem}\label{Section2}
The dynamics include the following standard features:
\begin{enumerate}
\item A reinsurance policy is a $\mathcal{B}\pr{\mathbb{R}_+}$-measurable function $u:\mathbb{R}_+\rightarrow\mathbb{R}_+$ such that $u(x)\leq x,\forall x\in\mathbb{R}_+$, modelling the part of the claim retained by the initial insurance company. The class of such policies is generally denoted by $\mathcal{R}$; see hereafter for different classes to which our study applies;
\item Given a reinsurance policy $u\in\mathcal{R}$ and non-negative claims $C_i$, $i\geq 1$, where $\pr{C_i}_{i\geq1}$ are independently and identically distributed (i.i.d.) with distribution function $F$, one computes the reinsurance-modified distribution
\[F^u(x):=\int_{\mathbb{R}_+}\mathbf{1}_{u(y)\leq x}F(dy)\] i.e. the distribution function of $u \pr{C_1}$;
\item Given a reinsurance policy $u$ and a reserve level and claim-dependent premium function $p:\mathbb{R}_+\times\mathbb{R}_+\rightarrow\mathbb{R}_+$, we define the premium
\[p^u(x):=\int_{\mathbb{R}_+}p\pr{y,x}dF^u(y)=\int_{\mathbb{R}_+}p\pr{u(y),x}F(dy).\]We note that, in the multiplicative case $p(y,x)=p(x)y$ and full retention $u^0(y)=y$, one gets $p^{u^0}(x)=p(x)\mathbb{E}\pp{C_1}=:\tilde{p}(x)$ (where $\mathbb{E}\pp{C_1}$ denotes the average of a single claim which is assumed to be finite).
\end{enumerate}
\subsection{Basic Assumptions}
\begin{assumption}
Throughout the paper, the following are assumed to hold true. \\
\textbf{[A1$_p$]}: The function $p:\mathbb{R}_+^2\rightarrow\mathbb{R}_+$ is non-decreasing with respect to both arguments. \\
\textbf{[A2$_p$]}: The function $p$ is uniformly continuous and
\begin{equation*}
\begin{split}
&\norm{p}_0 :=\int_0^\infty p(y,0)F(dy)<\infty,\\
& \pp{p}_1 :=\sup_{\substack{x,x'\in\mathbb{R}_+,  \\  x\neq x' }} \frac{ \displaystyle{ \int_0^\infty } \sup_{0\leq y\leq z} \abs{p(y,x)-p\pr{y,x'}}dF(z)}{\abs{x-x'}}<\infty.
\end{split}
\end{equation*}
\end{assumption}
\begin{remark}
\begin{enumerate}
\item As a consequence of \textbf{[A1$_p$]}, the average satisfies \[p^u(x)\leq \bar{p}(x):=\int_{\mathbb{R}^+} p(y,x)dF(y),\ \forall u\in\mathcal{R},\ \forall x\in\mathbb{R}_+.\]
\item
\textbf{[A2$_p$]} is satisfied if, for example, $p\pr{\cdot,0}$ has polynomial growth, the generic claim $C$ has a finite moment of the degree of such polynomial and $p$ is Lipschitz-continuous in $x$ uniformly in the first variable.
\item We consider the flow $\displaystyle{\bar{x}^x_t =x+\int_0^t\bar{p}\pr{\bar{x}^x_s}ds}$. Under the assumptions \textbf{[A1$_p$]}, \textbf{[A2$_p$]}, the application $\mathbb{R}_+\ni x\mapsto \bar{x}^x_t$ satisfies \[\bar{x}^x_t\leq\pr{x+t\norm{p}_0}e^{\pp{p}_1t}.\] It follows that, for $q>\pp{p}_1$,
$\displaystyle{\int_0^\infty e^{-qt}\bar{x}^x_tdt\leq Ax+B},\textnormal{ for some }A,B\in\mathbb{R}$. The arguments developed throughout the paper work as well under the uniform continuity of $p$ and this kind of linear growth assumption. We have chosen to limit ourselves to the above-mentioned assumptions in order to give better readability to the linearisation arguments.
\end{enumerate}
\end{remark}
\subsection{The Model and Immediate Properties}
 We consider the following dynamics for the reserve of an insurance company
\begin{equation*}
\label{Eq0}
X_t=x+\int_0^tp^{u_s}\pr{X_s}ds-\sum_{i=1}^{N_t}u_{\tau_i}\pr{C_i},
\end{equation*}
where $u_t \in \mathcal{R}$ is the reinsurance policy and $\tau_i$ stands for the time of $i$-th claim. We consider a rich enough probability space $\pr{\Omega,\mathcal{F},\mathbb{P}}$ supporting the (independent) Poisson process $N$ (of intensity $\lambda>0$) and the family of i.i.d. random variables (r.v.) $\pr{C_i}_{i\geq 1}$ modelling the claims. On this space we consider the natural right-continuous, $\mathbb{P}$-completed filtration $\mathbb{F}$.
\begin{enumerate}[label = \arabic*)]
\item Having fixed a triplet $\pi:=\pr{u,L,I}$ describing mixed continuous/singular strategies as explained before, the surplus process is given (under the $\mathbb{P}_x$), by the following equation
\begin{equation}
\label{Surpluspi}
X_t^\pi:=x+\int_0^t p^{u_s}\pr{X_s^\pi}ds-\sum_{i=1}^{N_t}u_{\tau_i}\pr{C_i}-L_t+I_t.
\end{equation}
\item The \textit{reinsurance} (or, more precisely, retention) \textit{policy} $u$ is an $\mathcal{R}$-valued predictable process, where, similar to \cite{azcue2005optimal}, $\mathcal{R}$ is one of the following:
\begin{itemize}
\item$\mathcal{R}_A$ (all Borel functions $u:\mathbb{R}_+\rightarrow\mathbb{R}_+$ such that (s.t.) $u(y)\leq y,\ y\geq 0$),
\item$\mathcal{R}^{a_0}_P$ (proportional retention $u(y)=u(1)y$, $u(1)\in\pp{a_0,1}$ with minimal proportion $a_0\in \pp{0,1}$),
\item $\mathcal{R}_{XL}$ (excess-of-loss policies $u(y)=\min\set{y,a}$, $a\geq 0$).
\end{itemize}
\item A right-continuous with left-limits (abbreviated to c\`{a}dl\`{a}g),  adapted process  $L$ such that $L_{0-}=0$ and which is non decreasing will be referred to as a \textit{dividend strategy};
\item A non-decreasing, c\`{a}dl\`{a}g, adapted process $I$ such that $I_{0-}=0$ will model the \textit{capital injection}.
\item Both $L$ and $I$ are such that $\displaystyle{\mathbb{E}_x \pp{\int_0^\infty e^{-\pp{p}_1s}\phi_sds}<\infty }$, where $\phi\in\set{I,L}$.
\item A \textit{strategy} is given by $\pi:=\pr{u,L,I}$ with components as before. The set of strategies of this type is $\Pi^+(x)$.
\item \begin{itemize}\item(before bankruptcy) for $t\geq 0$, the dividends are assumed to comply with
\begin{equation}
\label{RestrDividends}
\bigtriangleup L_t:=L_{t}-L_{t-}\leq X_{t-}^{\pi}-\bigtriangleup \bar{N}^u_t+\bigtriangleup I_t, \textnormal{ where }\bar{N}^u_t:=\sum_{i=1}^{N_t}u_{\tau_i}\pr{C_i}.
\end{equation}
\item (after \textit{bankruptcy} intervening as $X_{t-}^\pi-\bigtriangleup \bar{N}^u_t+\bigtriangleup I_t<0$), we freeze the components i.e.  $I_s=I_{t-}$, and $L_s=L_{t-},\ \forall{s\geq t}.$
\end{itemize}
\item The triplet $\pi\in\Pi^+(x)$ obeying to the restriction (7) is called an \textit{admissible strategy} and they form the class $\Pi(x)$.  Furthermore, whenever $u$ is a fixed reinsurance policy, we consider the section $\Pi_u(x):=\set{\pr{L,I}:\pi:=\pr{u,L,I}\in\Pi(x)}.$
\item For every such strategies $\pi\in\Pi(s)$,
\begin{enumerate}
\item we consider the bankruptcy time $\sigma_{0-}^{\pi}:=\inf\set{t\geq 0:\ X_{t-}^\pi-\bigtriangleup \bar{N}^u_t+\bigtriangleup I_t<0}$. This has the following significance. When a large claim intervenes and capital injection is considered too expensive, then bankruptcy is the solution to adopt. Whenever needed, the initial condition of $X^\pi$ (i.e. $x$) will be explicit in $\sigma_{0-}^{x,\pi}$;
\item The gain associated to such policies is
\begin{equation}
\label{Costpi}
v\pr{x,\pi}:=\mathbb{E}_x\pp{\int_0^{\sigma_{0-}^{\pi}}e^{-qs}\pr{dL_s-kdI_s}},
\end{equation}
where $x$ denotes the initial value of process $X$. The quantity $k\geq 1$ is the cost of one unit of injected capital;
\item The optimal gain is given by
\begin{equation}
\label{Value}
V(x):=\sup_{\pi\in\Pi(x)}v\pr{x,\pi},\ \forall x\in\mathbb{R}.
\end{equation}
\end{enumerate}
\end{enumerate}
\begin{remark}\label{RemSimultaneousIL}
If $\pi\in\Pi(x)$ is an admissible strategy, then, prior to $\sigma_{0-}^\pi$, it is sub-optimal to have $\bigtriangleup I_t \wedge \bigtriangleup L_t>0$. Indeed, the same dynamics (hence, the same bankruptcy time) and a better gain is got for $\tilde{I}_t:=I_t-\sum_{s\leq t\wedge\sigma_{0-}^\pi}\bigtriangleup L_s\wedge\bigtriangleup I_s,\ \tilde{L}_t:=L_t-\sum_{s\leq t\wedge\sigma_{0-}^\pi}\bigtriangleup L_s\wedge\bigtriangleup I_s$.
The same argument can be applied to the continuous parts.
\end{remark}
We have the following immediate properties.
\begin{proposition}
\label{PropBasicPropertiesXV}
\begin{enumerate}[label=\arabic*)]
\item Under the assumptions \textbf{[A1$_p$]} and \textbf{[A2$_p$]}, for every $x\in\mathbb{R}_+$, given a strategy $\pi\in\Pi^+ (x)$, the equation (\ref{Surpluspi}) admits a unique solution.
\item For any $x\in\mathbb{R}_+$, the family $\Pi(x)$ is not empty. When the initial positions $x',x\in\mathbb{R}_+$ satisfy $x\leq x'$, one has the inclusion $\Pi(x)\subset\Pi\pr{x'}$. Furthermore, for $\pi\in \Pi(x)$, one has, $\ \mathbb{P}$-a.s., $\sigma_{0-}^{x,\pi}\leq \sigma_{0-}^{x',\pi}$.
\item For all initial positions $x\in\mathbb{R}_+$, all strategies $\pi\in\Pi(x)$, and every $q\geq \pp{p}_1$,
\[\left\lbrace\begin{split}\max\set{X_{t-}^\pi,0}&\leq xe^{qt}+\norm{p}_0\frac{e^{qt}-1}{q}+\int_{\left[0,t \right)}e^{q\pr{t-s}}\pr{dI_s-dL_s}\\\max\set{X_t^\pi,0}&\leq xe^{qt}+\norm{p}_0\frac{e^{qt}-1}{q}+\int_{\pp{0,t}}e^{q\pr{t-s}}\pr{dI_s-dL_s}\end{split}\right.,\ \forall t\in\pp{0,\sigma^{\pi}_{0-}},\ \mathbb{P}-a.s.,\]on $\set{\sigma^{\pi}_{0-}<\infty}$.
\item  For arbitrary $\pi\in\Pi(x)$, one has the lower bound $x+\frac{\norm{p}_0}{\lambda+q}\leq V(x)$.
\item If $q\geq\pp{p}_1$, then, for every $\pi\in\Pi(x)$, one has the upper bound $v(x,\pi)\leq x+\frac{\norm{p}_0}{q}\leq x+\frac{\norm{p}_0}{\pp{p}_1}$.
\item For every $0<\varepsilon$,  one has $V\pr{x+\varepsilon}-V(x)\in\pp{0,k\varepsilon}$.
\end{enumerate}
\end{proposition}
\hspace{1.5em}The proof is quite standard and follows from the assumptions on $p$. Item 3. follows from comparison with the explicit solution of the no-jump linear equation driven by $y\mapsto\norm{p}_0+qy$ and with singular controls $L,I$. Item 4. is got by using the no-injection, no-reinsurance, pay all reserve (and declare bankruptcy at first claim) strategy.  The upper bound in (5) follows from Proposition \ref{PropBasicPropertiesXV}-(3).  Finally, the Proposition \ref{PropBasicPropertiesXV}-(6) which is useful to get the absolutely continuity and derivative bounds of $V$ is obtained by tempering with the singular controls $L$ and $I$ at $0$ which allows to pass from $x$ to $x+\varepsilon$ or the other way around.\\

The following quantitative properties will prove useful in order to qualify the occupation measures associated to controlled trajectories and the larger set of constraints appearing in the linearized formulation of our initial control problem.
\begin{proposition}\label{PropEstimILX}Let $x\geq 0$ and $\pi:=\pr{u,L,I}\in\Pi(x)$ be $1$-optimal (in the sense that the associated cost $v(x,\pi)\geq V(x)-1$). Then, for every $t\geq 0$,
\begin{enumerate}[label=\arabic*)]
\item the injection process satisfies $\displaystyle{\mathbb{E}_x\pp{\int_{\pp{0,t\wedge\sigma_{0-}^{\pi}}}e^{-qs}dI_s}\leq \frac{2\norm{p}_0+\pp{p}_1}{(k-1)\pp{p}_1}};$
\item the dividend process satisfies $\displaystyle{\mathbb{E}_x\pp{\int_{\pp{0,t\wedge\sigma_{0-}^{\pi}}}e^{-qs}dL_s}\leq x+\frac{(k+1)\norm{p}_0+\pp{p}_1}{(k-1)\pp{p}_1}};$
\item the associated solution satisfies $\displaystyle{\mathbb{E}_x\pp{\underset{t\geq 0}\sup\ e^{-q\pr{t\wedge\sigma_{0-}^\pi}}\pr{X^\pi_{t\wedge\sigma^{\pi}_{0-}}\vee 0}}\leq x+\frac{(k+1)\norm{p}_0+\pp{p}_1}{(k-1)\pp{p}_1} };$
\item Furthermore, $\displaystyle{\mathbb{E}_x\pp{\int_{\pp{0,\sigma_{0-}^\pi}} e^{-qt}\pr{X^\pi_{t}\vee 0}dt}\leq \frac{1}{q-\pp{p}_1}\pr{x+\frac{(k+1)\norm{p}_0+\pp{p}_1}{(k-1)\pp{p}_1}}}.$
\end{enumerate}
\end{proposition}
The proof is postponed to the Appendix.
\section{The Hamilton-Jacobi-Bellman Variational Inequality}\label{Section3}
As it is by now standard in this kind of control problems, we consider the Hamilton-Jacobi integro-differential inequality
\begin{equation}
\label{HJB}
\left\lbrace
\begin{split}
&\max\set{1-V'(x),V'(x)-k,H\pr{x,V,V'(x)}}=0,\ x\in\mathbb{R}_+,\\
&\min\set{\max\set{1-V'(x),V'(x)-k},V(x)}=0,\ x\in\mathbb{R}_-\setminus\set{0}.
\end{split}\right..
\end{equation}
Here, the Hamiltonian $H$ is defined as
\begin{equation}
\label{Hamiltonian}
H\pr{x,\psi,a}:=\sup_{u\in\mathcal{R}}\pr{-\pr{\lambda+q}\psi(x)+p^u(x)a+\lambda\int_{\mathbb{R}_+}\psi(x-u(y))dF(y)},
\end{equation}for continuous $\psi$ and we also define \[\mathcal{L}^{u}\psi(x)=p^u(x)\psi'(x)+\lambda\int_{\mathbb{R}_+}\psi(x-y)dF^u(y)-\pr{\lambda+q}\psi(x),\]when $\psi$ is absolutely continuous.
Let us give a rather immediate characterization of $V$ in connection to the above system
\begin{proposition}\label{Vsupersol}
The function $V$ is non negative, absolutely continuous ($\mathcal{AC}$) and a super-solution to \eqref{HJB}\footnote{In the sense that for Lebesgue-almost every $x\in\mathbb{R}$ the function $\phi$ has a derivative $\phi'(x)$ and the equation \eqref{HJB} is satisfied at $x$. In particular, note that $V$ is non-decreasing.}such that $V(x)\leq x+\frac{\norm{p}_0}{q}$, for all $x\geq 0$, and $V(0)\geq \frac{\norm{p}_0}{\lambda+q}>0$.
\end{proposition}
The proof is immediate and standard. It follows from the dynamic programming principle and it is sketched in the Appendix.
On $\mathbb{R}_+$, one only needs to show that $\mathcal{L}^uV\leq 0$ and this follows from an infinitesimal argument at $t\wedge \tau_1$ (the time of the first claim) for no-dividend, no-injection strategy. For negative reserves, the two possibilities are to declare bankruptcy (leading to $0$-value) or to inject capital to reach $0$ position and continue with the best strategies from there. This yields $V(x)=\pr{V(0)+kx}^+$.\\
We proceed with the following basic remark.
\begin{remark}
Let $\phi$ be a super-solution to \eqref{HJB} which is non negative and absolutely continuous. The function $\phi$ is zero for arguments $x\leq -\phi(0)$ (in order to comply with the equation on the negative arguments). As a consequence, there exists $x_0\in\mathbb{R}$ for which $\phi(x)=0,\ \forall x\leq x_0$ and Lebesgue-a.e. on $(x_0,\infty)$, $\phi$ admits a derivative in $\pp{1,k}$. Moreover, please note that $\phi$ is non-decreasing.\end{remark}
Before proceeding with the linear programming arguments, let us give the following useful result.
\begin{proposition}\label{PropSupersol}Let $\phi$ be a super-solution to \eqref{HJB} which is non negative and absolutely continuous and such that
\begin{enumerate}
\item[(a)] $\phi(0)>0$;
 \item[(b)] $\phi$ has a linear growth and $\phi(y)$ is upper bounded by $c+y$, for some constant $c\in\mathbb{R}$ and $y\geq 0$.
\end{enumerate} Furthermore, let \[\mathbb{R}_-^*\ni x_0:=\sup\set{x\in\mathbb{R}:\ \phi(x)=0}.\]
\begin{enumerate}
\item There exists a family of non-decreasing functions $\pr{\phi_n}_{n\in\mathbb{N}}$ of class $C^1$ s.t.
\begin{enumerate}
\item $\phi(x)\leq \phi_n(x)\leq \phi(x)+\frac{\tilde{c}}{n}$, for all $x\in\mathbb{R}$;
\item $1\leq \phi'_n(x)\leq k$ on $\left[x_0,\infty\right)$,
\item ${\phi_n'}$ converges (as $n\rightarrow\infty)$ to $\phi'$, pointwise a.e. on $\left[x_0,\infty\right)$, and
\item $\mathcal{L}^u\phi_n(x)\leq 0$, for all $x\geq 0$.
\end{enumerate}
\item Moreover, $V(x)\leq \phi(x)$, for every $x\in\mathbb{R}_+$.
\end{enumerate}
Here, $\tilde{c}$ is a generic constant independent of $n\geq 1$ and $x\in \mathbb{R}$.
\end{proposition}
The proof is postponed to the Appendix.
\begin{remark}\label{dPsi}
\begin{enumerate}\item A careful look at the proof shows that the functions $\psi_n$ and, hence, $\phi_n$ have $k$-upper-bounded derivatives on the entire state space $\mathbb{R}$.
\item When $\psi_n$ are constructed from $V$ (recall that $V(x)=\pr{V(0)+kx}^+$ on $\mathbb{R}_-$), one has $\phi_n'(y)=k,\ \forall y\in\pp{-\frac{V(0)}{k},\frac{-2}{n}}$.
\end{enumerate}
\end{remark}

\section{Linear Programming Arguments}\label{Section4}
\subsection{Occupation Measures}
Whenever $L$ and $I$ are as before, we consider \[\mathcal{J}(\omega):=\mathcal{J}^{L}(\omega)\cup\mathcal{J}^{I}(\omega):=\set{s\in\mathbb{R}_+:\ \Delta L_s\pr{\omega}\neq 0}\cup\set{s\in\mathbb{R}_+:\ \Delta I_s(\omega)\neq 0},\  \omega\in\Omega,\] and
\[\tilde{\mathcal{J}}(\omega):=\set{s\in\mathbb{R}_+:\ \Delta L_s\pr{\omega}+\Delta I_s(\omega)\neq 0}\cup\set{\tau_i(\omega): \ i\geq 1},\  \omega\in\Omega,\]where $\tau_i$ designates the time of arrival for the $i^{th}$ claim ($i\geq 1$). These constitute the times at which the reserve $X^{x,\pi}$ has a jump.
Furthermore, we distinguish four non-decreasing processes corresponding to the continuous and discrete part of $L$, resp. $I$ and their associated $q$-discounted measures $L^c$, $L-L^c$, $I^c$, $I-I^c$.
We apply It\^o's formula to (the finite variation process) $s\mapsto e^{-qs}\phi\pr{X_s^{\pi}}$ on $\pp{0,t\wedge\sigma_{0-}^\pi}$ and owing to Remark \ref{RemSimultaneousIL}, we get \begin{equation}\label{Ito}\begin{split}&\mathbb{E}_x\pp{e^{-q\pr{t\wedge\sigma_{0-}}}\phi\pr{X_{t\wedge\sigma_{0-}}^{\pi}}}\\
=&\phi(x)+\mathbb{E}_x\pp{\int_0^{t\wedge\sigma_{0-}^\pi}\pr{-qe^{-qs}\phi(X^\pi_{s-})+e^{-qs}\phi'\pr{X_{s-}^{\pi}}p^{u_s}\pr{X_{s-}^{\pi}}}ds}\\&-\mathbb{E}_x\pp{\int_0^{t\wedge\sigma_{0-}^\pi}e^{-qs}\phi'\pr{X_{s-}^{\pi}}\pr{dL_s^c-dI_s^c}}\\&+\mathbb{E}_x\pp{\sum_{s\in\pp{0, t\wedge\sigma_{0-}^\pi}\cap\tilde{\mathcal{J}}}e^{-qs}\pr{\phi\pr{X_{s-}^{\pi}-\Delta\bar{N}_s^{u_{s}}-\Delta L_s+\Delta I_s}-\phi\pr{X_{s-}^{\pi}}}}\\
=&\phi(x)\!+\!\mathbb{E}_x \!\pp{\int_0^{t\wedge\sigma_{0-}^\pi}\!\!e^{-qs} \! \pr{\lambda\!\int_{\mathbb{R}_+}\!\!\phi\!\pr{X_{s-}^\pi \!-\!z}F^{u_s}(dz)\!-\!\pr{q\!+\!\lambda}\phi\pr{X_{s-}^\pi}\!+\! \phi'\pr{X_{s-}^{\pi}}p^{u_s}\pr{X_{s-}^{\pi}} \!}\!ds}\\
&-\mathbb{E}_x \!\! \pp{\!\int_{0}^{t\wedge\sigma_{0-}^\pi}e^{-qs}\!\phi'\pr{X^\pi_{s-}}\!\pr{dL_s^c \!-\!dI_s^c}} \!\!+\mathbb{E}_x \!\!\pp{\sum_{s\in \mathcal{J}^I\cap{\pp{0,t\wedge\sigma_{0-}^\pi}}} \!\!e^{-qs}\!\int_0^{\Delta I_s}\! \!\phi'\pr{X_{s-}^{\pi}\!-\!\Delta\bar{N}_s^{u_{s}} \!+\! z}dz}\\&-\mathbb{E}_x\pp{\sum_{s\in \mathcal{J}^{L}\cap{\pp{0,t\wedge\sigma_{0-}^\pi}}}e^{-qs}\int_0^{\Delta L_s}\phi'\pr{X_{s-}^{\pi}-\Delta\bar{N}_s^{u_{s}}-z}dz}.
\end{split}\end{equation}
In view of this formula, we consider, for $x\geq0$, and admissible policies $\pr{u,L,I}\in\Pi(x)$ and for the stopping times $\sigma\in\mathcal{T}:=\set{t\wedge\sigma_{0-}:\ t\geq 0}$ the following occupation measures
\begin{enumerate}
\item[(i)] $\gamma_0\in\mathcal{P}\pr{\mathbb{R}_+\times\mathbb{R}}$, \[\gamma_0\pr{ds_1dy_1}=\mathbb{P}_x\pr{\sigma\in ds_1,\ X^\pi_\sigma\in dy_1};\]
\item[(ii)] $\gamma_1\in\mathcal{M}^+\pr{\mathbb{R}_+\times\mathbb{R}\times\mathbb{R}_+\times \mathbb{R}_+\times\mathcal{R}}$,\[\gamma_1\pr{ds_1dy_1ds_2dy_2du}:=
\mathbb{E}_x\pp{\mathbf{1}_{\set{\sigma\in ds_1,\ X^\pi_{\sigma}\in dy_1}}e^{-qs_2}\mathbf{1}_{\set{s_2\in \pp{0,s_1}, \ X^\pi_{s_2-}\in dy_2},\ u_{s_2}\in du} ds_2};\]
\item[(iii)]$\gamma_2\in\mathcal{M}^+\pr{\mathbb{R}_+\times\mathbb{R}\times\mathbb{R}_+\times \mathbb{R}_+\times\mathcal{R}\times \mathbb{R}_+}$,
\begin{align*}
&\gamma_2\pr{ds_1dy_1ds_2dy_2dudl}:=\mathbb{E}_x\pp{\mathbf{1}_{\set{\sigma\in ds_1,\ X^\pi_{\sigma}\in dy_1}}\mathbf{1}_{\set{s_2\in \pp{0,s_1}, \ X^\pi_{s_2-}\in dy_2}}e^{-qs_2}L^c\pr{ds_2}},\\
&+\mathbb{E}_x \!\! \pp{\mathbf{1}_{\set{\sigma\in ds_1,\ X^\pi_{\sigma}\in dy_1}} \!\int_{\mathbb{R}_+} \! \! \mathbf{1}_{\set{X^\pi_{s_2-}-\Delta\bar{N}_{s_2}^{u}-z\in dy_2,\ z\leq l,\ u_{s_2}\in du}}dz\delta_{\Delta L_{s_2}}(dl)e^{-qs_2}\pr{\sum_{s\in\mathcal{J}\cap\pp{0,s_1}} \!\! \delta_{s}\pr{ds_2} \!}\!};
\end{align*}
\item[(iv)]$\gamma_3\in \mathcal{M}^+\pr{\mathbb{R}_+\times\mathbb{R}\times\mathbb{R}_+\times \mathbb{R}_+\times\mathcal{R}\times \mathbb{R}_+}$,
\begin{align*}&\gamma_3\pr{ds_1dy_1ds_2dy_2dudi}:=\mathbb{E}_x\pp{\mathbf{1}_{\set{\sigma\in ds_1,\ X^\pi_{\sigma}\in dy_1}}\mathbf{1}_{\set{s_2\in \pp{0,s_1}, \ X^\pi_{s_2-}\in dy_2}}e^{-qs_2}I^c\pr{ds_2}} \\
&+\mathbb{E}_x \!\! \pp{\mathbf{1}_{\set{\sigma\in ds_1,\ X^\pi_{\sigma}\in dy_1}} \! \int_{\mathbb{R}_+} \!\! \mathbf{1}_{\set{X^\pi_{s_2-}-\Delta\bar{N}_{s_2}^{u}+z\in dy_2,\ z\leq i,\ u_{s_2}\in du}}dz\delta_{\Delta I_{s_2}}(di) e^{-qs_2}\pr{\sum_{s\in\mathcal{J}\cap\pp{0,s_1}} \!\! \delta_{s}\pr{ds_2} \!}\!}.
\end{align*}
\end{enumerate}
A look at the properties in Proposition \ref{PropEstimILX}, the definitions of the above measures and It\^{o}'s formula in \eqref{Ito} yield the following structural properties (for 1-optimal policies).
\begin{equation}\label{StrProperties}
\setlength{\abovedisplayskip}{3pt}
\setlength{\belowdisplayskip}{3pt}
\begin{split}
(\rm{i})\ &\int_{\mathbb{R}_+\times\mathbb{R}}e^{-qs_1}\pr{y_1\vee 0}\gamma_0\pr{ds_1dy_1}\leq x^++\frac{(k+1)\norm{p}_0+\pp{p}_1}{(k-1)\pp{p}_1};\\
(\rm{ii})\ &\left\lbrace\begin{split}(a)\ &\gamma_1\in \mathcal{M}^+_{\frac{1}{q}}\pr{\mathbb{R}_+\times\mathbb{R}\times\mathbb{R}_+\times \mathbb{R}_+\times\mathcal{R}};\\
(b)\ & \gamma_1\pr{ds_1dy_1,\mathbb{R}_+\times \mathbb{R}_+\times\mathcal{R}}=\frac{1-e^{-qs_1}}{q}\gamma_0\pr{ds_1dy_1};\\
(c)\ & \int_{\mathbb{R}_+\times\mathbb{R}\times\mathbb{R}_+\times \mathbb{R}_+\times\mathcal{R}}y_2\gamma_1\pr{ds_1dy_1ds_2dy_2du}\leq \frac{1}{q-\pp{p}_1}\pr{x^++\frac{(k+1)\norm{p}_0+\pp{p}_1}{(k-1)\pp{p}_1}};
\end{split}\right.\\
(\rm{iii})\ &\gamma_2\in\mathcal{M}^+_{x^++\frac{(k+1)\norm{p}_0+\pp{p}_1}{(k-1)\pp{p}_1}}\pr{\mathbb{R}_+\times\mathbb{R}\times\mathbb{R}_+\times \mathbb{R}_+\times\mathcal{R}\times \mathbb{R}_+};\\
(\rm{iv})\ &\gamma_3\in\mathcal{M}^+_{(-x)^++\frac{2\norm{p}_0+\pp{p}_1}{(k-1)\pp{p}_1}}\pr{\mathbb{R}_+\times\mathbb{R}\times\mathbb{R}_+\times \mathbb{R}_+\times\mathcal{R}\times \mathbb{R}_+};\\
(\rm{v}) \ &\textnormal{For every regular test function with linear growth } \phi\in C^1\pr{\mathbb{R};\mathbb{R}},\\
&\left.\begin{split}\int_{\mathbb{R}_+\times\mathbb{R}} e^{-qs_1}\phi\pr{y_1}\gamma_0\pr{ds_1,dy_1}=&\phi(x)+\int_{\mathbb{R}_+\times\mathcal{R}} \mathcal{L}^u\phi\pr{y_2}\gamma_1\pr{\mathbb{R}_+\,\mathbb{R},\mathbb{R}_+,dy_2,du}\\&+\int_{\mathbb{R}_+} \phi'\pr{y_2}\pr{\gamma_3-\gamma_2}\pr{\mathbb{R}_+,\mathbb{R},\mathbb{R}_+,dy_2,\mathcal{R},\mathbb{R}_+}.\end{split}\right.
\end{split}
\end{equation}
\begin{remark}\label{RemStructure}
\begin{enumerate}
\item Whenever $t=0$, the occupation measure $\gamma_0$ is a Dirac one $\delta_{\pr{0,x}}$, $\gamma_1$ is null (thus, still obeys to (ii)(b)), while $\gamma_2$ and $\gamma_3$ account for the jump in $L$ and $I$ at $0$.
\item Whenever the occupation measure is associated to $\sigma=\sigma_{0-}^\pi\wedge t$ for $t>0$, the reader will note that the support of $\gamma_0$ (and similarly for the others, of course) satisfies $Supp\pr{\gamma_0}\subset \pp{0,t}\times \mathbb{R}$.
\item The identification of marginals for $\gamma_1$ in (ii)(b) together with the fact that the measures $\gamma_0$ do not have atoms at $0$ (again, provided that $t>0$) explains why we call the first measure $\gamma_0$.
\item  We have stated the conditions with $x^+$ instead of $x$ (although we have assumed that $x\geq 0$) for reasons that will appear obvious in Section \ref{Subsect2Stg}. The same applies to, \eqref{StrProperties} (iv).
\end{enumerate}
\end{remark}
\hspace{1.5em} From now on, in order to avoid lengthy notations, we will write $\displaystyle{\int e^{-qs_1}\pr{y_1\vee 0}d\gamma_0}$ instead of $\displaystyle{\int_{\mathbb{R}_+\times\mathbb{R}}e^{-qs_1}\pr{y_1\vee 0}\gamma_0\pr{ds_1dy_1} }$ and so on.
\begin{definition}
\label{Theta}
The set of constraints for initial positions $x\geq 0$, designated hereafter by $\Theta(x)$ is defined by
\begin{equation*}
\Theta(x)=\set{\begin{split}\gamma=\pr{\gamma_0,\gamma_1,\gamma_2,\gamma_3}:\
&\gamma_0\in\mathcal{P}\pr{\mathbb{R}_+\times\mathbb{R}};\ \gamma_1\in\mathcal{M}^+\pr{\mathbb{R}_+\times\mathbb{R}\times\mathbb{R}_+\times \mathbb{R}_+\times\mathcal{R}};\\ &\gamma_2,\gamma_3\in \mathcal{M}^+\pr{\mathbb{R}_+\times\mathbb{R}\times\mathbb{R}_+\times \mathbb{R}_+\times\mathcal{R}\times \mathbb{R}_+},\\
&\gamma \textnormal{ satisfies } \eqref{StrProperties}
\end{split}}.
\end{equation*}
For every $t\geq 0$, the set of all measures $\gamma\in\Theta(x)$ such that their first marginal(s) have support included in $\pp{0,t}\times \mathbb{R}$ will be designated by $\Theta_t(x)$.
\end{definition}
\subsection{The Duality Result}
\begin{theorem}\label{ThmDual}
The linearized function $\Lambda:\mathbb{R}_+\longrightarrow\mathbb{R}_+^*$ defined by\[\Lambda(x):=\sup_{\gamma\in\Theta(x)}\pr{\gamma_2-k\gamma_3}\pr{\mathbb{R}_+\times\mathbb{R}\times\mathbb{R}_+\times \mathbb{R}_+\times\mathcal{R}\times \mathbb{R}_+},\]coincides with $V$ on $\mathbb{R}_+$ and the common value is given by the dual value
\begin{equation}
\begin{split}
\label{Dual}
V(x)=\Lambda(x)=\Lambda^*(x):=&\inf\ \phi(x)\\ &\begin{split} s.t.\ &\phi\in C^1\pr{\mathbb{R};\mathbb{R}_+},\textnormal{ with linear growth, }\phi'\in\pp{1,k}\textnormal{ on }\left[0,\infty\right);\\
&\mathcal{L}^u\phi(y)\leq 0,\ \forall u\in\mathcal{R},\ \forall y\in\mathbb{R}_+.\end{split}
\end{split}
\end{equation}
Moreover, the functions in the description of $\Lambda^*$ can be taken with $k$-upper bounded derivative on $\mathbb{R}$ i.e. \begin{equation}\label{Dualk}
\begin{split}
\Lambda^*(x)=&\inf\ \phi(x)\\ &\begin{split} s.t.\ &\phi\in C^1\pr{\mathbb{R};\mathbb{R}_+},\ \phi'\leq k,\textnormal{ on }\mathbb{R},\ \phi'\geq 1,\textnormal{ on }\left[0,\infty\right);\\
&\mathcal{L}^u\phi(y)\leq 0,\ \forall u\in\mathcal{R},\ \forall y\in\mathbb{R}_+.\end{split}
\end{split}
\end{equation}
\end{theorem}
The reader is referred to the Appendix for the proof of this result.
\section{Two (and Multi)-stage Linear Decision Formulations and the Dual Programming Algorithm}\label{Section5}
\subsection{The Linearized Dynamic Programming Principle and A Two-Stage Decision Problem}\label{Sub5.1}
Let us take a look at the dynamic programming principle and provide some linearized versions taking into account the previous dual formulation(s). Prior to giving our main result, let us extend the functions $\Lambda$ and $\Lambda^*$ from $\mathbb{R}_+$ to the entire real axis $\mathbb{R}$ by setting\[\Lambda(x)=\Lambda^*(x)=\pr{\Lambda(0)+kx}^+,\ x<0.\]
\begin{corollary}\label{Cor}
For every $t>0$ and every $x\geq 0$, the value function $V$ can be retrieved as a two-stage problem
\begin{equation}\label{TwoStage}
\begin{split}
V(x)& =
\sup_{\gamma\in\Theta_t(x)}\set{\int \pr{d\gamma_2-kd\gamma_3}+\int e^{-qs_1}{\Lambda}^*\pr{y_1}d\gamma_0}\\
&=\sup_{\gamma\in\Theta_t(x)}\set{\int \pr{d\gamma_2-kd\gamma_3}+ \inf_{\phi \in \mathbf{F}}\  \int e^{-qs_1}\phi(y_1)d\gamma_0} ,
\end{split}
\end{equation}
where $\mathbf{F} := \set{\phi\in C^1\pr{\mathbb{R};\mathbb{R}_+} \big|\  \phi'\leq k,\textnormal{ on }\mathbb{R},\ \phi'\geq 1,\textnormal{ on }\left[0,\infty\right);\
\mathcal{L}^u\phi(y)\leq 0,\ \forall u\in\mathcal{R},\ \forall y\in\mathbb{R}_+}$.
\end{corollary}
\begin{remark}
The reader will note that the first equality provides the classical dynamic programming principle (please also take a look at Proposition \ref{DPP} in the Appendix. The only particularity is the "linearized" formulation. The second equality, however, no longer asks for minimization at all points $y_1$ but includes a linear criterion involving part of $\gamma$.
\end{remark}
Again, for readability purposes, we postpone the proof to the Appendix.
\begin{remark}
The reader will easily note that the proof shows that both sets $\Theta_t (x)$ can be replaced with $\Theta (x)$.
\end{remark}

\subsection{Extension of the Set of Constraints. Coherence With the DPP}\label{Sub5.2}
As we have seen in the previous subsection, in order to provide a linear formulation of the dynamic programming principle, we needed to extend the value function to negative initial positions. In the same spirit, before proceeding with our two-stage algorithm, let us give a natural extension of the (constraint) sets of measures $\Theta_t$ to initial data $x$ that are negative.
\begin{enumerate}
\item For $x <-\frac{V(0)}{k}$,  it is optimal to declare bankruptcy, and, for such $x$, we set \begin{equation}\label{ThetaExt1}
\Theta_t (x) :=\set{
\begin{split}\gamma=\pr{\gamma_0 ,\gamma_1 ,\gamma_2 ,\gamma_3}:\ &\gamma_0\in\mathcal{P} \!\pr{\mathbb{R}_+ \!\times\mathbb{R}};\ \gamma_1\in\mathcal{M}^+\pr{\mathbb{R}_+ \!\times\mathbb{R} \!\times\mathbb{R}_+ \!\times \mathbb{R}_+ \!\times\mathcal{R}};\\ &\gamma_2,\gamma_3\in \mathcal{M}^+\pr{\mathbb{R}_+\times\mathbb{R}\times\mathbb{R}_+\times \mathbb{R}_+\times\mathcal{R}\times \mathbb{R}_+},
\\
&\gamma_0 = \delta_0 \otimes \delta_{x} , \gamma_1 =0,\ \gamma_2 =\gamma_3 =0    \end{split}}.
\end{equation}
It is easy to check that (\ref{StrProperties}) still holds true:
\begin{itemize}
\item (i) follows easily for the Dirac measure: $\displaystyle{\int e^{-qs_1}(y_1\vee 0)d\gamma_0=0}$;
\item (ii) (a) and (c), (iii) and (iv) are mass or moments bounds which are obvious for the zero measures;
\item (ii) (b) is true since $1-e^{-qs_1}=0,\ \gamma_0-a.s.$;
\item Finally, (v) reads $\displaystyle{\int_{\mathbb{R}_+\times\mathbb{R}} e^{-qs_1}\phi(y_1)\delta_{\pr{0,x}}(ds_1dy_1)=\phi(x) }$ which is trivial.
\end{itemize}
\item For $x \in \big[-\frac{V(0)}{k} ,0\big)$, one naturally defines
\begin{equation}\label{ThetaExt2}\begin{split}
&\delta^{\#}_{-x}\pr{dy_2}:=\int_{\mathbb{R}_+} \mathbf{1}_{ x+z\in dy_2,\ z\leq -x}dz;\\
&\Theta_t (x): = \set{\begin{split}\gamma =\pr{\gamma_0 ,\gamma_1 ,\gamma_2 ,\gamma_3}&:=\pr{\gamma'_0 ,\gamma'_1 ,\gamma'_2,\gamma'_3 +\delta_{(0,0,0)}\otimes \delta^{\#}_{-x}\otimes\delta_0\otimes\delta_{-x} } ,\\ &\text{where}\ \pr{\gamma'_0 ,\gamma'_1 ,\gamma'_2 ,\gamma'_3}\in\Theta_t (0)\end{split}}.\end{split}
\end{equation}
Then \eqref{StrProperties} holds true:
\begin{itemize}
\item (i), (ii), (iii) hold true for initial data $0$, hence for $x^+=0$;
\item To see (iv), one notes that \[\int 1d\gamma_3=\int 1d\gamma_3'+\int_0^{-x}dz\leq \frac{2\norm{p_0}+\pp{p}_1}{(k-1)\pp{p}_1}-x,\]due to the choice of $\gamma_3'$ and the bounds in $\Theta_t(0)$;
\item Finally, in order to prove (v), for $\phi\in C^1\pr{\mathbb{R};\mathbb{R}}$, one has
\begin{align*}
&\int e^{-qs_1}\phi(y_1)d\gamma_0=\int e^{-qs_1}\phi(y_1)d\gamma_0'\\
&=\phi(0)+\int\mathcal{L}^u\phi(y_2)d\gamma_1+\int\phi'(y_2)(d\gamma_3-d\gamma_2)-\int\phi'(y_2)\delta_{(0,0,0)}\otimes \delta^{\#}_{-x}\pr{dy_2}\otimes\delta_0\otimes\delta_{-x} \\
&=\int\mathcal{L}^u\phi(y_2)d\gamma_1+\int\phi'(y_2)(d\gamma_3-d\gamma_2)+\phi(0)-\int_0^{-x}\phi'(x+z)dz\\
&=\int\mathcal{L}^u\phi(y_2)d\gamma_1+\int\phi'(y_2)(d\gamma_3-d\gamma_2)+\phi ( x).
\end{align*}
\end{itemize}
\end{enumerate}
Using this, we get a simple extension of the formula (\ref{TwoStage}) to the entire real axis.
\begin{proposition}\label{PropNegx}

\begin{enumerate}
\item The equality (\ref{TwoStage}) is valid for all $x\!\in\!\mathbb{R}$, where $\Theta_t $ is extended using \eqref{ThetaExt1}, \eqref{ThetaExt2}.
\item Furthermore, the dual formulation \eqref{Dualk} holds true for $x<0$, i.e.
\begin{equation*}
\begin{split}
\Lambda(x)=\Lambda^*(x):=&\inf\ \phi(x)\\ &\begin{split} s.t.\ &\phi\in C^1\pr{\mathbb{R};\mathbb{R}_+},\ \phi'\leq k,\textnormal{ on }\mathbb{R},\ \phi'\geq 1,\textnormal{ on }\left[0,\infty\right);\\
&\mathcal{L}^u\phi(y)\leq 0,\ \forall u\in\mathcal{R},\ \forall y\in\mathbb{R}_+.\end{split}
\end{split}
\end{equation*}
\end{enumerate}
\end{proposition}
The explicit proof is relegated to the Appendix.\\
As a by-product of the first assertion, we get the following multi-stage linear problem (written here for two time steps $t_1$ and $t_2$ but generalizable to a sequence of such steps).
\begin{corollary}\label{CorollaryMultiStage}
Let $\set{t_i:\ 1\leq i\leq 2}\subset \mathbb{R}_+$. Then, for every $x\in\mathbb{R}$, it holds \begin{align*}
&\Lambda^*(x)\\
=&\sup_{\gamma^1\in\Theta_{t_1}(x)}\set{\int\pr{d\gamma^1_2-kd\gamma^1_3}
+\int\sup_{\gamma^2\in\Theta_{t_2}(y_1^1)}\set{\int\pr{d\gamma^2_2-kd\gamma^2_3}+\int e^{-qs_1^2}\Lambda^*(y_1^2)d\gamma^2_0}e^{-qs_1^1}d\gamma^1_0}.
\end{align*}
\end{corollary}
With $\scal{\cdot,\cdot}$ denoting the duality product between the space of measures and continuous functions,  the previous relation can be seen as a linear infinite-dimension problem.

\subsection{The Two-Stage Algorithm}\label{Subsect2Stg}\label{Sub5.3}
Next, with the help of the above considerations, we will give the two-stage algorithm to solve the optimal control problem (see Corollary \ref{CorollaryMultiStage}). The algorithm includes two parts: forward simulation and backward recursion.
\subsubsection{Space Considerations and some techniques in algorithm}
\begin{enumerate}
\item ({\em Space reduction and pseudo-compactness}) For reasons that will appear clear in the next subsection, we consider the measures to be set on some compact space i.e. to be supported by\[(s_1,y_1,y_2,i,l)\in\pp{0,T}\times \pp{-\frac{\norm{p}_0}{k\pp{p}_1},x^{\max}}\times \pp{0,x^{\max}}\times \pp{0,i^{\max}}\times \pp{0,l^{\max}}.\]
\begin{enumerate}
\item We consider the optimization problem for initial capital $x\leq \bar{x}$.
\item The dynamic programming principle (cf. \eqref{TwoStage}) and the multi-stage one (cf. Corollary \ref{CorollaryMultiStage}) will be applied with $t\leq T<\infty$.
\item A quick look at Proposition \ref{PropEstimILX}, assertion 3. shows that, up to a fixed error $\varepsilon>0$, \[\sup_{x\leq \bar{x}}\mathbb{P}_x\pr{\sup_{t\leq T}X^\pi_t\geq x^{\max,\varepsilon}}\leq e^{qT}\frac{\bar{x}+\frac{(k+1)\norm{p}_0+\pp{p}_1}{(k-1)\pp{p}_1}}{x^{\max,\varepsilon}}\leq \varepsilon,\]with the obvious choice $x^{\max,\varepsilon}:=\frac{\bar{x}+\frac{(k+1)\norm{p}_0+\pp{p}_1}{(k-1)\pp{p}_1}}{\varepsilon e^{-qT}}$. In particular, one can work with measures supported by $y_1\in\pp{\frac{-V(0)}{k},x^{\max,\varepsilon}}\subset\pp{-\frac{\norm{p}_0}{k\pp{p}_1},x^{\max,\varepsilon}},\ y_2\in \pp{0,x^{\max,\varepsilon}}$ and control the loss.
\item One essentially expects the injection of capital to intervene as a rescue measure, usually as the reserve becomes negative and appearing as impulsive feedback, see \cite{avram2020equity}. This instantaneous injection should not exceed the upper-bound for the expected reward. Under this assumption,  and using Proposition \ref{PropBasicPropertiesXV}, assertion 5.,  the measure $\gamma_3$ is supported on \[i\leq \frac{V(0)}{k}+x^{\max,\varepsilon}\leq i^{\max,\varepsilon}:=x^{\max,\varepsilon}+\frac{\norm{p}_0}{k\pp{p}_1}.\]
\item Similar arguments apply to dividends: instantaneous dividends cannot exceed the reserve (hence $x^{\max,\varepsilon}$) while continuous dividends are usually related to keeping the reserve at some level (thus a fraction of the premium). It is natural to consider \[ l\leq l^{\max,\varepsilon}:=\pp{p}_1x^{\max,\varepsilon}+\norm{p}_0.\]
\end{enumerate}
\item Standard arguments (cf.  \cite[Chapter 1]{llavona1986approximation}) allow to replace continuous test functions in $\Theta_t(x)$ by combinations of monomials (via Nachbin's Theorem) of degree up to some fixed positive integer;
\end{enumerate}
$\blacktriangleright$ ({\em SOS technique}) The test functions appearing in $\sup\inf$ formulation \eqref{TwoStage} satisfy $\phi\geq 0$ on the associated Euclidean space.  Hence, we will use the long-established Sum-of Squares method (SOS for short).
\begin{enumerate}
\item The notation $\sum_{2r} [y]$ stands for the polynomials of degree not exceeding $2r$ in $y\in\mathbb{R}^N$, for some positive integer $N$ and obtained as sum of squares. A polynomial $p(y) \in \sum_{2r} [y]$ if there exists a finite family of polynomials $\xi_1 (y) ,\cdots,\xi_n (y)\in \mathbb{R}_r [y]$ such that $p(y)=\sum_{i=1}^{n} \xi_i (y)^2$.
\item Asking that a polynomial $p\in\sum_{2r} [y]$ is equivalent to the existence of a symmetric positive semi-definite matrix $\mathbf{M}$ (denoted by $\mathbf{M}\succeq 0$) such that $p(y)=\mathbf{X}^T \mathbf{M} \mathbf{X}$ where, written as a $\alpha$-indexed vector, \[\mathbf{X} = \pr{\underset{1\leq i\leq N}{\Pi}y_i^{\alpha_i}:\ \alpha_i\in\mathbb{N},\norm{\alpha}:=\sum_{1\leq i\leq N}\alpha_i\leq r}.\]This leads to an optimization problem with choices $\mathbf{M}$ satisfying the constraint expressed as linear matrix inequalities (LMI) and roughly taking the form $\mathbf{M} \succeq 0$.
\end{enumerate}
$\blacktriangleright$ ({\em truncated quadratic module}) For a given semi-algebraic set $\mathbf{S}:= \set{y\in\mathbb{R}^N\ | \ h_{i} (y) \geq 0,\ i\in I }$ where $h_i$, $i\in I$, are polynomials, the truncated quadratic module of degree $r$ of this set is defined as
\[  Q_r (\mathbf{S}) := \sigma_0^2 (y) +\sum_{i\in I} \sigma_i^2 (y) h_i (y)   , \]
where $\sigma_0^2 (y),\ \sigma_i^2 (y) \in \sum_{2r} [y]$, with the restriction that  $\text{deg} \pr{\sigma_i^2 h_i} \leq 2r$. Such polynomials are obviously non-negative for all $y\in\mathbf{S}$ and this method will be used as slacks in the algorithm. \\
$\blacktriangleright$ ({\em LMI-relaxation}) Next, we will recall a method named LMI (linear matrix inequality)-relaxations which will be used in the following algorithm (readers are referred to \cite{lasserre2005nonlinear} and \cite{henrion2021moment} for more details).
\begin{enumerate}[label= (\alph*)]
\item Define $\mathcal{D}_{r}^{y} \subset \mathbb{R} [t,y]$ (resp. $\mathcal{D}_{r}^{t}$) be the monomials of the canonical basis of the space of polynomials in the variables $y$ (resp. $t$) and of total degree less than $r$. In practice, it is required that the support of the problem $\mathbf{Y} := \set{y\in\mathbb{R}^N\  | \ \theta_j (y) \geq 0,\ j\in J}$ be a compact semi-algebraic subset of $\mathbb{R}^N$, for some finite index set $J$ and polynomials $\theta_j (y)$. Assume $\mu$ is a measure supported on $\pp{0,T} \times\mathbf{Y}$ ($\mu^t$ supported on $\pp{0,T}$; $\mu^y$ supported on $\mathbf{Y}$ ) and $\displaystyle{m_{\alpha\beta}:=\int t^{\beta}y^{\alpha} d\mu }$, for $\norm{\alpha}+\beta \leq 2r$ ($\displaystyle{m^t_{\beta}:=\int t^{\beta} d\mu^t }$; $\displaystyle{m^y_{\alpha}:=\int y^{\alpha} d\mu^y}$ ) is moment of the measure $\mu$ ($\mu^t$; $\mu^y$). We emphasize that $\alpha=\pr{\alpha_1,\ldots, \alpha_N}$ is an integer-valued multi-index, we denote by $y^\alpha:=y_1^{\alpha_1}\ldots y_N^{\alpha_N}$, and, again,  by $\norm{\alpha}:=\underset{1\leq i\leq N}{\sum}\alpha_i$.
Moreover, $\mathbf{m}$, $\mathbf{m}^t$ and $\mathbf{m}^y$ are vectors of the corresponding moments $m_{\alpha\beta}$, $m^t_{\beta}$ and $m^y_{\alpha}$.
\item  From Putinar's Positivstellensatz (\cite{putinar1993positive} Theorem 3.3 ), there exists a sufficient condition to guarantee that the moments $m_{\alpha\beta}$ have a representing measure $\mu$ supported on $[0,T]\times\mathbf{Y}$ i.e.
\begin{align*}&L_{\mathbf{m}} (h^2) \geq 0,\ \forall h\in \mathbb{R}_r [t,y];  \ L_{\mathbf{m}^y} (\theta_j h^2) \geq 0, \forall h\in \mathcal{D}_{r-\lceil \frac{\text{deg}\  \theta_j}{2} \rceil}^{y}; \\
 &L_{\mathbf{m}^t} (t(T-t) h^2) \geq 0, \ \forall h\in \mathcal{D}_{r-1}^{t} ; \ y\in\mathbb{R}^N.\end{align*}
Here, $L_{\mathbf{m}} (h^2):= \int h^2 (t,y) d\mu = \int \sum_{\norm{\alpha}+\beta\leq 2r} h_{\alpha \beta} t^{\beta} y^{\alpha} d\mu  =\sum_{\norm{\alpha}+\beta\leq 2r} h_{\alpha \beta} m_{\alpha\beta}$
is a linear functional on $\mathbb{R} [t,y]$. Similar definitions hold for $L_{\mathbf{m}^y}$ and $L_{\mathbf{m}^t}$) and $h_{\alpha\beta}$ denotes the coefficient of function $h^2$.
\item   Next, we will use moment matrix defined in (\cite{lasserre2008nonlinear} Lasserre et al. 2008) to further deal with the conditions in (b). The detailed definition of moment matrices $M_{r} \pr{\mathbf{m}}$ will be presented in the Appendix. Let $\mathbf{h}$ denote the vector of coefficients of the polynomial $h$ of degree less than $r$, then
\begin{equation*}
\begin{split}
 &\scal{\mathbf{h} , M_{r} \pr{\mathbf{m}} \mathbf{h}} = L_{\mathbf{m}} (h^2) =\int h^2 d\mu \geq 0,\ \forall h\in\mathbb{R} [t,y]; \\
 & \scal{\mathbf{h} , M_{r-\lceil \frac{\text{deg}\theta_j }{2} \rceil} \pr{\theta_j \mathbf{m}^y} \mathbf{h}} = L_{\mathbf{m}^y} (\theta_j h^2) =\int \theta_j h^2 d\mu^y \geq 0,\ \forall h\in \mathcal{D}_{r}^{y},\ \forall j\in J;  \\
 & \scal{\mathbf{h} , M_{r-1} \pr{t(T-t) \mathbf{m}^t} \mathbf{h}} = L_{\mathbf{m}^t } (t(T-t) h^2) =\int t(T-t) h^2 d\mu^t \geq 0,\ \forall h\in \mathcal{D}_{r}^{t}. \\
\end{split}
\end{equation*}
Then, the above inequalities imply that the localizing matrices $M_{r} \pr{\mathbf{m}}$, $M_{r-\lceil\frac{\text{deg}\theta_j}{2}\rceil} \pr{\theta_j \mathbf{m}^y }$ and $M_{r-1} \pr{t(T-t) \mathbf{m}^t}$ ) are symmetric non-negative matrices, i.e. for every r,
$ M_{r} \pr{\mathbf{m}} \succeq 0$, $\ M_{r-\lceil \frac{\text{deg}\theta_j}{2}\rceil} \pr{\theta_j\mathbf{m}^y} \succeq 0$ and $ M_{r-1} \pr{t(T-t) \mathbf{m}^t }\succeq 0$.
\end{enumerate}

\subsubsection{Notations for moments}
In this section, to avoid the complexity of notations, we use $\mathbf{Y}$ to represent the compact set in which the $y$ position is located, and specify the constructions as if the dimension $N=1$. Of course, while this is valid for the marginal measures $\int_{\pp{0,T}}e^{-qt}\gamma_0(dt,dy)$ (concerned only with the occupation of the space at bankruptcy), they are a bit lengthier for the other measures. The reader is invited to refer to the aforementioned space considerations to understand what $\mathbf{Y}$ stands for in the case of measures $\gamma_i,\ i\in\set{1,2,3}$.

\begin{enumerate}
\item ({\em $\bar{\mathbb{Y}}_{0}$: moments of $\bar{\gamma}_{0}$})
For a given order $r\in\mathbb{N}^*$, $\phi (y)\in C^1 \pr{\mathbf{Y} ;\mathbb{R}_+}$ can be approximated by the following polynomial
\begin{equation}\label{poly}
\varphi (y) =  \sum_{\beta \leq 2r}  C_{\beta}   y^{\beta} , \quad  \text{where}\ m, \beta \in\mathbb{N}, \  C_{\beta} \in \mathbb{R} .
\end{equation}

We define a (discounted) version of $\gamma_0$, i.e. $\bar{\gamma}_0 (dt,dy) := e^{-qt} \gamma_0 (dt,dy)$ and define the following moment of $\bar{\gamma}_{0}$,
\[  \bar{Y}_{0}^{ \beta } := \int_{\mathbf{Y}} y^{\beta} \bar{\gamma}_0 ([0,T], dy) = \int_{[0,T]\times\mathbf{Y}} e^{-qt} {y}^{\beta} \gamma_{0} (dt, dy) ,\  \text{with} \ \beta \leq 2r  , \]
and $\bar{\mathbb{Y}}_{0}$ is the $\beta$-indexed vector of the above moments.
Thus,
\begin{equation*}
\int_{[0,T]\times\mathbf{Y}}  e^{-qt} \phi (y) d\gamma_{0} =  \int_{[0,T]\times\mathbf{Y}}   \sum_{\beta \leq 2r}  C_{\beta}  {y }^{\beta} e^{-qt} d\gamma_{0} = \sum_{\beta \leq 2r}  C_{\beta} \bar{Y}_{0}^{\beta }
=: L_{\bar{\mathbb{Y}}_{0}} \pr{\phi (y)}.
\end{equation*}
Time moments are defined similarly and so are mixed time and space moments with $2$-dimensional multi-indexes.
\item ({\em $\mathbb{Y}_{i}$: moments of $\gamma_{i}$, $i=1,2,3$}) We define the following moments of the measures ($\gamma_0$, $\gamma_{1}$, $\gamma_{2}$, $\gamma_{3}$), for which the time is well present compared with $\bar{\mathbb{Y}}_{0}$,
\[   Y_{i}^{\beta} := \int_{\mathbf{Y}}  y^{\beta} d\gamma_{i}, \ \text{for}\ i=1,2,3;\ \norm{\beta}\leq 2r,  \]
and $\mathbb{Y}_{i}$ denotes the vector of corresponding moment. For instance, in the case of $i=1$, one has a dimension $N=5$ and the set $\mathbf{Y}:=\pp{0,T}\times \pp{-\frac{\norm{p}_0}{k\pp{p}_1},x^{\max}}\times\pp{0,T}\times\pp{0,x^{\max}}\times \bar{R}$. The $y$ variable stands for $\pr{s_1,y_1,s_2,y_2,u}$. To render $\bar{R}$ convenient, please take a look at Example \ref{ExpProp}. Then, $\beta\in \mathbb{N}^5$ is a $5$-dimensional multi-index.
Thus, for a given polynomial of type \eqref{poly} (this time, $\beta\leq 2r$ has to be read as $\norm{\beta}\leq 2r$), one has
\begin{equation*}
 \int_{\mathbf{Y}}   \phi (y) \gamma_{j} (dy) = \sum_{\norm{\beta} \leq 2r}  C_{\beta} Y_{j}^{\beta }
=: L_{\mathbb{Y}_{j}} \pr{ \phi (y)},\ j=0,1,2,3.
\end{equation*}
Of course, the spaces $\mathbf{Y}$ should be different for indexes $j$, but they can be embedded into a larger, universal one.
Of particular importance are polynomials $\phi$ that only depend on the $y_2$ component.
\item ({\em $\mathbb{D}$: moment sof Dirac measure}) The moments of Dirac measure are defined as $D_x^{\alpha} :=\int  y^{\alpha} \delta_x (dy). $
 Then, $\mathbb{D}_x $ is the corresponding vector of moments. The definition of linear operators $L_{\mathbb{D}_x}$ is similar to the one of $L_{\mathbb{Y}_{j}}$, with $j\in\set{0,1,2,3}$.
\end{enumerate}

\begin{remark}The reader will keep in mind that although we talk about $\mathbf{Y}$, it is not just a one-dimensional setting.  As a matter of fact, we can consider polynomials $\varphi\in\mathbb{R}\pp{s_1,y_1,y_2,u,l,i}$. Let us further explain Putinar's result \cite[Remark after Eq. (16)]{putinar1993positive} and the meaning it has for our constraints.
\begin{enumerate}
\item Consistency of $\gamma_0$: $y_1\in\pp{-\frac{\norm{p}_0}{k\pp{p}_1},x^{\max,\varepsilon}}$, $t\in\pp{0,T}$ such that one imposes \begin{equation}
\label{Consist0}
L_{\bar{\mathbb{Y}}_0} (\varphi^2)\geq 0,\ L_{\bar{\mathbb{Y}}_0}\pr{\pr{y_1+\frac{\norm{p}_0}{k\pp{p_1}}}\pr{x^{\max,\varepsilon}-y_1} \varphi^2}\geq 0, \ L_{\bar{\mathbb{Y}}_0}\pr{t(T-t)\varphi^2}\geq 0.
\end{equation}Moreover, $M_r \pr{\bar{\mathbb{Y}}_0}$ stands for the moment matrix of moment $\bar{\mathbb{Y}}_0$ with respect to $\bar{\gamma}_0$.
\item Let us explain the consistency of $\gamma_2$ (same to $\gamma_3$), or, to be more precise, of the marginal $\bar \gamma_2:=\int_{\mathbb{R}_+}\gamma_2\pr{\cdot,\cdot,ds_2,\cdot,\cdot}$ since the behaviour of $s_2$ is of little interest. One would have
 \begin{equation}
\label{Consist2}
\begin{cases}
&L_{\mathbb{Y}_2}(\varphi^2)\geq 0,\ L_{\mathbb{Y}_2}\pr{\pr{y_1+\frac{\norm{p}_0}{k\pp{p_1}}}\pr{x^{\max,\varepsilon}-y_1}\varphi^2}\geq 0, \ L_{\mathbb{Y}_2}\pr{t(T-t)\varphi^2}\geq 0,\\
&L_{\mathbb{Y}_2}\pr{y_2\pr{x^{\max,\varepsilon}-y_2}\varphi^2}\geq 0,\ L_{\mathbb{Y}_2}\pr{l\pr{l^{\max,\varepsilon}-l}\varphi^2}\geq 0.
\end{cases}
\end{equation}
In this case, $\mathbb{Y}_2$ is constructed from a multi-index $\alpha=\pr{\alpha_i}_{1\leq i\leq 4}\subset \mathbb{N}^4$ and defining its length as $\norm{\alpha}:=\sum_{1\leq i\leq 4}\alpha_i$.  For such index, one sets $\displaystyle{Y_2^{\alpha}:=\int {s_1}^{\alpha_1}y_1^{\alpha_2}y_2^{\alpha_3}l^{\alpha_4}d\gamma_2}$.

\end{enumerate}
\end{remark}

\subsubsection{The Forward Step}
In the remaining of the paper, we will use $\mathbb{Y}_{j,z}$ to denote the moment with respect to the occupation measure $\gamma_{j,z}$, $j\in \set{0,1,2,3}$,  where $z\in \mathbb{N}$ denotes the $z$-th iteration of the algorithm.

We need to point out that we do not explicitly mention \eqref{StrProperties}, i-iv and similar conditions in the constraints below for space reasons. This will be taken into account in the actual optimization procedure. With this convention, the forward step appears as
\begin{equation}\label{forward1}
\begin{split}
F_{z} &= \max_{\pp{\mathbb{Y}_{j,z}}_{j=0}^{3}} \set{ L_{\mathbb{Y}_{2,z}} (1) -kL_{\mathbb{Y}_{3,z}} (1) +L_{\bar{\mathbb{Y}}_{0,z}} \pr{ \hat{\phi}_{z-1} (y_1)}} ,\\
&\quad s.t.\  L_{\mathbb{D}_x } (y^{\alpha}) + L_{\mathbb{Y}_{1,z}} \pr{\mathcal{L}^u (y_2^{\alpha})} +\alpha L_{\mathbb{Y}_{3,z} -\mathbb{Y}_{2,z}} \pr{y_2^{\alpha -1}} =L_{\bar{\mathbb{Y}}_{0,z}} \pr{y_1^{\alpha}} \\
&\qquad\qquad \text{with\ } \max \set{ \text{deg}\pr{p^u (y)} +\alpha-1 ,\ \alpha}\leq 2r ,\ u\in\mathcal{R},\\
&\qquad\ \  M_r \!\pr{\mathbb{Y}_{j,z}} \!\succeq \! 0,
\ \!M_{r-1} \!\pr{s_1 (t_1 \!-\!s_1) \mathbb{Y}_{j,z} (s_1)} \!\succeq \! 0,\\
 &\qquad\ \ M_{r-1} \! \pr{ \pr{x^{\max}-y_1}\pr{y_1+\frac{\norm{p}_0}{k\pp{p}_1}} \mathbb{Y}_{j,z} (y_1)} \!\succeq \! 0,\  j\!=\!0,\! 1,\! 2, \!3, \\
&\qquad\ \ M_{r-1} \pr{ l(l^{max}-l) \mathbb{Y}_{2,z} (l) } \succeq 0,\  M_{r-1} \pr{ i(i^{max}-i) \mathbb{Y}_{3,z} (i) } \succeq 0, \\
&\qquad\ \  M_{r-1} \pr{\pr{x^{\max}-y_2}y_2\mathbb{Y}_{j,z} (y_2) } \succeq 0,\ j\!=1,2,3, \\
\end{split}
\end{equation}
Here, $ L_{\bar{\mathbb{Y}}_{0,z}} \pr{\hat{\phi}_{z-1} (y_1)} $ is the cutting plane of $\displaystyle{\int e^{-qs_1} \Lambda^* \pr{y_1} d\gamma_{0,z} }$ obtained from its dual formulation where $\hat{\phi}_{z-1}$ is the optimal function in the previous backward step.

Our attentive reader has certainly noticed that some of the constraints (those referring to the integro-differential operator $\mathcal{L}^u$, resp., $p^u$) might fail to be polynomial.  However, the approximation of these terms strongly relies on the type of policies envisaged. Let us focus on the proportional case and we give an example to explain.
\begin{example}\label{ExpProp}
The usual distribution $F$ of claims is an exponential law, i.e. $F(d\tilde{y} )= \kappa e^{-\kappa \tilde{y}} \mathbf{1}_{\mathbb{R}_+}(\tilde y)d\tilde{y}$. We give an SOS polynomial approximation of $p\pr{u(\hat{y}),y}$, i.e. $\displaystyle{p\pr{u(\hat{y}),y} := \sum_{a+b \leq 2r'}  C_{ab} \pr{u(\hat{y})}^{a} y^b}$,\footnote{The coefficients $C_{ab}$ are subject to additional LMI constraints not stated explicitly here.} for some order $r'
\in \mathbb{N}$ Then, in proportional reinsurance case:
\begin{equation*}
\begin{split}
p^u (y) &=\int_{\mathbb{R}_+} p\pr{ u(\hat{y}) ,y} F(d\hat{y}) =\int_{\mathbb{R}_+} \sum_{a+b \leq 2r'} C_{ab} \pr{u(1)\hat{y}}^a y^b \kappa e^{-\kappa \hat{y}} d\hat{y}     \\
& =\sum_{a+b \leq 2r'} C_{ab}\  y^b \pr{u(1)}^a \int_{\mathbb{R}_+} \hat{y}^a e^{-\kappa \hat{y}} d\hat{y} = \sum_{a+b \leq 2r'} C_{ab} \frac{\pr{u(1)^a}}{\kappa^{a}} \Gamma (a+1) y^b  , \\
\end{split}
\end{equation*}
where $\Gamma$ is Gamma function; and
\begin{equation*}
\begin{split}
\int_{\mathbb{R}_+} \!\! \pr{y-\tilde{y}}^{\alpha} dF^u (\tilde{y}) =\!\!\int_{\mathbb{R}_+}\! \pr{y-u(\tilde{y})}^{\alpha} \kappa e^{-\kappa \tilde{y}} d\tilde{y} =\!\!\int_{\mathbb{R}_+} \!\pr{y-u(1)\tilde{y}}^{\alpha} \kappa e^{-\kappa \tilde{y}} d\tilde{y} =\sum_{i=0}^{\alpha} (-1)^i \frac{A_{\alpha}^{i} \pr{u(1)}^i}{\kappa^{i}} y^{\alpha -i},
\end{split}
\end{equation*}
where $A_{\alpha}^{i}:=\frac{\alpha !}{i!\pr{\alpha-i}!}$.  As a consequence,
\begin{align*}
\mathcal{L}^u x^\alpha=&p^u(x)\alpha x^{\alpha-1}+\lambda\int_{\mathbb{R}_+}\pr{x-y}^\alpha dF^{u}(y)-\pr{\lambda+q}x^\alpha\\
\approx &\sum_{a+b \leq 2r'} C_{ab} \frac{\alpha\pr{u(1)}^a}{\kappa^{a}} \Gamma (a+1) x^{b+\alpha-1}+\lambda \sum_{a=1}^{\alpha} (-1)^a \frac{A_{\alpha}^{a} \pr{u(1)}^a}{\kappa^{a}} y^{\alpha -a}-q x^\alpha
\end{align*}gives, a $2r'+r-1$-degree polynomial approximation of the infinitesimal generator as claimed (where the test functions have degree $\alpha\leq r$).  Furthermore, in this case, the optimization is sought over the real parameter $u(1)\in\pp{0,1}$.
\end{example}

Since the polynomial forms of $p^u (y)$ and the integral of $dF^u$ depend on the choice of reinsurance policy $u$, we will still use them to represent the polynomial form later for the convenience of presentation. Therefore, our first-stage forward problem is a purely polynomial form of moments over occupation measures.

The key to solving the forward problem in real programming lies in the realization of the occupation measures. Even though there is a way to directly generate a measure by using GloptiPoly 3 which is a package dealing with simple occupation measure, our framework is obviously much more complicated than the one in Lasserre's \cite{henrion2009gloptipoly}, \cite{lasserre2005nonlinear} (see the following Remark \ref{diracgamma0}). In this paper, we will use the method of Dirac sampling to realize our occupation measure.

For example, to obtain $\gamma_{0,z}$ in our first stage problem of forward step, we will firstly discretize the compact spaces they are in (i.e. $\pp{0,t_1}$ and $\pp{x^{min}:=\frac{\norm{p}_0}{k\pp{p}_1} ,x^{max}}$). Then, we will use Dirac function to take out the sample that falls in the grids generated by the fixed discretization. If we have $n$ sets of data (simulate the dynamic process (\ref{Surpluspi}) n times), we use Dirac sampling $n$ times, and then calculate the frequency of sample in each grid to get an approximate distribution which is our constructed $\gamma_{0,z}$. The other measures can be constructed in similar method but for the processes $L$, $I$.

\begin{remark}\label{diracgamma0}
First of all, $s_1$ can be skipped. Secondly,  if we are working on $\pp{0,t_1}$, then the discounted measure $e^{-qs_1}\gamma_{0,z}(ds_1,dy_1)$ is
\begin{itemize}
\item either $e^{-qt_1}\gamma_{0,z}(dy_1)\delta_{t_1}(ds_1)$, with $\gamma_{0,z}$ supported by $\pp{0,x^{\max}}$ . Then,  one samples $\pp{0,x^{\max}}$ by setting $\set{x_n:=\frac{nx^{\max}}{N}:\ 0\leq n\leq N}$ and define \[\Delta_1:=\set{e^{-qt_1}\delta_{ x_n  }(dy_1)\delta_{t_1}(ds_1),0\leq n\leq N}.\]
\item or $\gamma_{0,z}(\pp{0,t_1},dy_1)=\delta_{\set{-\frac{V(0)}{k}}}(dy_1)$ i.e. bankruptcy is declared when injection is no longer profitable. Of course, this is a short-cut as bankruptcy may intervene for more important losses. However, the value function is null prior to and at this point which can be regarded as a "cemetery" state when the trajectories and values get frozen.  Since $V(0)\in \pp{\frac{\norm{p}_0}{\lambda+q},\frac{\norm{p}_0}{q}}$,  one can sample, as before,  $x_n'$ over $\pp{-\frac{\norm{p}_0}{kq},-\frac{\norm{p}_0}{k(\lambda+q)}}$ and $t^n_1$ over $\pp{0,t_1}$ and get a set \[\Delta_2:=\set{e^{-qt_i^1}\delta_{\pr{t^i_1,x_j'}}(ds_1,dy_1):\ 1\leq i,j\leq N}. \] Then $e^{-qs_1}\gamma_{0,z}(ds_1,dy_1)$ is well approximated by convex combinations over $\Delta:=\Delta_1\cup\Delta_2$.
\item for the second case, one should also envisage plainly using $\delta_{-\frac{\hat \phi_{z-1}(0)}{k}}$ as $y_1$-marginal, whenever $\phi_{z-1}(0)$ provides a guess that will be clear hereafter.
\end{itemize}
In other words, having fixed $N$, generating $\gamma_{0}$ comes to generating some $(N+1)^2$ coefficients $\alpha_{i,j}\geq 0,\ \sum_{i,j}\alpha_{i,j}=1$.
\end{remark}

Then, we solve, among the candidates generated, this forward problem and get the optimal moments $\widehat{\mathbb{Y}}_{j,z}$, $j=0,1,2,3$. Actually, if we get these moments, then we will act as if we got the corresponding optimal occupation measures $\hat{\gamma}_{i,z}$, $i=0,1,2,3$.  Moreover, if the measure $\hat{\gamma}_{0,z}$ is fixed, the time and dynamic process term $\pr{s_1 , y_1}$ can be obtained in the sense of distributions.
The lower bound of this iteration is $F_{LB,z} =L_{\widehat{\mathbb{Y}}_{2,z}} (1) -kL_{\widehat{\mathbb{Y}}_{3,z}} (1) + \displaystyle{\int} e^{-qs_1 } F_{z} (y_1) d\hat{\gamma}_{0,z}$.

\subsubsection{Prepared illustration of dual formulation }
Let us take a look into the changes operated by the previous considerations.  To this purpose, let us consider the dual formulation of the value function at some point $y_1\in \mathbb{R}$ (cf.  Theorem \ref{ThmDual}, Eq. \eqref{Dualk}; see also Proposition \ref{PropNegx} for $y_1<0$).
\begin{equation*}
\begin{split}
\Lambda^*(y_1)=&\inf\ \phi(y_1)\\ &\begin{split} s.t.\ &\phi\in C^1\pr{\mathbb{R};\mathbb{R}_+},\ \phi'\leq k,\textnormal{ on }\mathbb{R},\ \phi'\geq 1,\textnormal{ on }\left[0,\infty\right);\\
&\mathcal{L}^u\phi(y)\leq 0,\ \forall u\in\mathcal{R},\ \forall y\in\mathbb{R}_+.\end{split}
\end{split}
\end{equation*}
\begin{enumerate}
\item {\em The function $\phi$ and the bounds on derivatives.}
Our problem allows one to consider the functions $\phi$ restricted to the (compact semi-algebraic) domain of $y^1$ i.e. $\pp{x^{min}:=-\frac{\norm{p}_0}{k\pp{p}_1},x^{max}}$.
Having fixed $\varepsilon>0$,  the function $\Lambda^*$ is approximated by polynomials $\varphi$ (with an adequate degree $2r$ depending on $\varepsilon$).
\item Here, we give some explanations why the polynomial $\varphi (y) \in \sum_{2r} [y]$ approximating test functions $\phi (y)$ depend on $\varepsilon$. First of all, we assume that the polynomials $\varphi$ give a $C^1$-approximation of the smooth functions $\phi (x)$.
\begin{enumerate}

\item Given an acceptable error $\varepsilon >0$, $\sup_{y\in [ x^{min} ,x^{max}  ] } \abs{\pr{\phi' -\varphi' } (y)} \leq \varepsilon$,  $\varphi (0) =\phi (0) \geq \frac{\norm{p}_0}{\lambda +q}$, then one gets the derivative $\varphi'$ should satisfy the constraint $\varphi' \in \pp{1-\varepsilon ,k+\varepsilon}$ because $\phi' \in\pp{1,k}$. On the negative semi-axis, $\varphi'=k$ is a good choice (but than one has piecewise polynomials).
\item  Obviously, $ \abs{\phi (y) -\varphi (y)}  \leq \varepsilon \abs{y} $. For $\mathcal{L}^u \phi (y) \leq 0$, $y\in \pp{0, x^{max}} $, and $p^u (y) \leq \pp{p}_1 y+ \norm{p}_0 $, one has, for $y\in  \pp{0, x^{max}}$,
\begin{equation*}
\begin{split}
&\mathcal{L}^{u} \varphi (y)  \\
\leq &\mathcal{L}^u \phi (y) \!+\!p^u (y) \abs{\varphi' (y) \!-\! \phi' (y) }\! +\!\lambda \!\! \int_{0}^{\infty} \!\abs{\varphi (y\!-\!y') \!-\! \phi (y\!-\! y')}\! dF^u (y')\! +\! \pr{\lambda +q} \abs{\varphi (y) \!-\!\phi (y)}  \\
\leq &  \pr{ \pp{p}_1 y+\norm{p}_0  } \varepsilon +\! \lambda\varepsilon  \int_{0}^{\infty} \!\abs{y-u(y')} dF (y')  + \pr{\lambda \!+\!q} y\varepsilon
\leq  \big( \pr{  \pp{p}_1 +\! 2\lambda +q } x^{\max} +\! \norm{p}_0 \!\big) \varepsilon. \\
\end{split}
\end{equation*}
We have assumed here that $x^{\max}\leq \mathbb{E}\pp{C}$, in order not to complicate the notations. This is achievable for $\varepsilon>0$ small enough, by the definition of $x^{\max,\varepsilon}$.
\end{enumerate}
In order to ensure that polynomials $\varphi$ satisfy the necessary condition $\mathcal{L}^u (\varphi) \leq 0$, $y\in \pp{x^{min} ,x^{max}}$, we can define polynomials $\varphi^{\varepsilon} (y) =\varphi (y) +C\varepsilon  $ \footnote{ The polynomial $\varphi^{\varepsilon} (y)$ can be constructed by adding constant $C\varepsilon$ to the first constant term of $\varphi (y)$ that the power of $y$ is zero and not changing the rest of the terms of $\varphi (y) $.    }, where $C=\frac{\pr{  \pp{p}_1 + 2\lambda +q } x^{max} +\norm{p}_0 }{q}$, then one has $\mathcal{L}^u \varphi^{\varepsilon} (y) \leq 0$, for $y\in \pp{0, x^{max}} $.

\item For the convenience of the presentation, we write the approximating polynomial depending on $\varepsilon$ in form of  $\varphi (y):=\sum_{\beta_z \leq 2r} C_{\beta_z} y^{\beta_z}$, $\beta_z\in\mathbb{N} $, $ C_{\beta_z} $\footnote{The coefficient should satisfy the linear matrix inequalities introduced in SOS technique, cf. Section 5.3.1.} $ \in \mathbb{R}$ and the coefficient $C_0$ depend on $\varepsilon$.
\end{enumerate}

From now on, again for simplicity, we will use $\phi$ to denote the polynomial version exhibited before.

Because the dual problem will be solved as LP (linear programming) problem over finite  moments of occupation measures, we need to take the "inf" out of the integral in dual formulation. Hence, we introduce the following lemma.

\begin{lemma}\label{infchange}
From (\ref{Dual}) and Proposition \ref{PropNegx}, one has the following equality:
\[ \int e^{-qs_1} \Lambda^*(y_1) d\gamma_{0,z} = \inf_{\phi\in\mathbf{F}}\  \int e^{-qs_1}\phi(y_1)d\gamma_{0,z} , \]
and we recall that $\mathbf{F}\!=\!\set{\! \phi \!\in \! C^1 \!\pr{\mathbb{R};\mathbb{R}_+\!}\big| \phi'\leq k,\textnormal{on }\mathbb{R},\ \!\phi'\geq 1,\textnormal{on }\!\left[0,\infty\right);\mathcal{L}^u\phi(y)\!\leq \! 0,\ \forall u\in\!\mathcal{R},\ y\!\in\!\mathbb{R}_+}$.
\end{lemma}
Again, the details are provided in the Appendix.

\subsubsection{The Backward Step}
Next, we will give the backward recursion step.  Then, the dual formulation of this problem becomes an LP problem (cf.  Subsections 5.3.2 and 5.3.4): for a given admissible error $\varepsilon >0$,
\begin{equation}\label{back1}
\begin{split}
B_{z}=& \inf_{\phi} \int e^{-qs_1} \phi_z (y_1) d\hat{\gamma}_{0,z} = \min_{ \beta ,C_{\beta}} \sum_{\beta \leq 2r}  C_{\beta} \widehat{\bar{Y}}_{0,z}^{\beta}, \\
 s.t. \ &\sum_{\beta \leq 2r} \beta C_{\beta} y^{\beta -1} \geq 1-\varepsilon,\ \forall y \in \pp{0 ,x^{max}} ;\
\sum_{\beta \leq 2r} \beta C_{\beta} y^{\beta -1} \leq k+\varepsilon,\ \forall y \in \pp{x^{min} ,x^{max}}, \\
 &\sum_{\beta \leq 2r} C_{\beta} \set{ \beta p^{u} (y) y^{\beta -1}  + \lambda \int_{\mathbb{R}_{+}} (y-\tilde{y})^{\beta} dF^{u} (\tilde{y}) - (\lambda +q) y^{\beta}     } \leq 0,\  \forall y\in\mathbb{R}_+,\ u\in \mathcal{R}, \\
 & M_{r} \!\pr{\widehat{\bar{\mathbb{Y}}}_{0,z}} \!\succeq \! 0 , \ M_{r-1} \pr{s_1 (t_1 \!-\!s_1) \widehat{\bar{\mathbb{Y}}}_{0,z} (s_1)} \!\succeq\! 0, \ M_{r-1} \pr{ \theta_1 \widehat{\bar{\mathbb{Y}}}_{0,z}(y_1)} \!\succeq \! 0, \\
\end{split}
\end{equation}
where $\hat{\gamma}_{0,z}$ is the optimal measure obtained from the forward step and $\phi_z$ stands for the objective function of $z$-th backward step with measure $\hat{\gamma}_{0,z}$.

Moreover, in the real implementation, one can use another version of constraints which is with slacks. Here, we will use the method of truncated quadratic module (cf. section 5.3.1) and define the needed compact semi-algebraic sets $\mathbf{Y}_{z}:=\set{y\in\mathbb{R} |\ \theta_1 (y):=(y-x^{min})(x^{max} -y) \geq 0,\ y\varpropto \gamma_{0,z}}$ and $\mathbf{Y}_{z}^+ :=\set{y\in\mathbb{R} |\ \theta_2 (y):= y(x^{max} -y) \geq 0,\ y\varpropto \gamma_{0,z}}$. For a given occupation measure $\gamma_{0,z}$, the symbol ``$y\varpropto \gamma_{0,z}$'' means that $y$ take values of the specific associated $x_k$ from which the measure $\gamma_{0,z}$ has been sampled.


Therefore, the improved version of our backward problem with slacks is
\begin{equation}\label{Back-Dual 1}
\begin{split}
B_{z}=& \min_{\beta ,C_{\beta}} \sum_{\beta\leq2r }  C_{\beta} \widehat{\bar{Y}}_{0,z}^{ \beta}, \\
s.t. \ &\sum_{\beta \leq 2r} \beta C_{\beta} y^{\beta -1} -(1-\varepsilon )  = Q_r (\mathbf{Y}_{z}^+); \
\ k+\varepsilon - \sum_{\beta \leq 2r} \beta C_{\beta} y^{\beta -1} = Q_r (\mathbf{Y}_{z}),  \\
&\sum_{\beta \leq 2r} C_{\beta}  \set{- \beta p^{u} (y) y^{\beta -1}  - \lambda \int_{\mathbb{R}_{+}} (y-\tilde{y})^{\beta} dF^{u} (\tilde{y}) + (\lambda +q) y^{\beta}     }  = Q_r (\mathbf{Y}_{z}^+),\ \ u\in \mathcal{R}, \\
&\qquad \text{with}\ \max \set{\beta -1 +\text{deg}(p^u (y)) ,\ \beta }\leq 2r, \\
& M_{r} \pr{\widehat{\bar{\mathbb{Y}}}_{0,z}} \succeq 0 , \ M_{r-1} \pr{s_1 (t_1 -s_1) \widehat{\bar{\mathbb{Y}}}_{0,z} (s_1)} \succeq 0, \ M_{r-1 } \pr{ \theta_1 \widehat{\bar{\mathbb{Y}}}_{0,z} (y_1)} \succeq 0,  \\
\end{split}
\end{equation}
where the forth line of constraint conditions is to maintain the consistency with the order of $Q_r (\cdot)$. Then, we solve this LP problem, find the optimal coefficients $\hat{\beta}$, $\hat{C}_{\beta}$ i.e. we get the corresponding optimal function $\hat{\phi}_z (y) = \sum \hat{C}_{\beta} y^{\hat{\beta}} $ and it will be used in the newt forward step as a cutting plane, i.e. $L_{\bar{\mathbb{Y}}_{0,z+1}} \pr{ \hat{\phi}_{z} (y_1)}$.
By solving this dual problem, one will get the upper bound $B_{UB ,z} :=L_{\hat{\mathbb{Y}}_{z}} (1) -k L_{\hat{\mathbb{Y}}_{3,z}} (1) + B_{z}$.\\
Finally, let us give our complete two-stage algorithm.

\begin{tabular}{l}
\hline
\hline
\textbf{Algorithm: Two-Stage DDP}
\\ [0.5ex]
\hline
\textbf{Input:}  Functions $F$, $p$, initial value $x$, policy $u$, tolerance $\varepsilon$, intensity $\lambda$, order $r$.\\
\textbf{Output:} Upper bound $B_{UB,z}$, lower bound $F_{LB,z}$, measure $\gamma_{z}$. \\ [0.5ex]
\hline
Step (a): $z \leftarrow 1$; Initialize $\hat{\phi}_{0} =0$.\\
Step (b): \textbf{Forward simulation:} \\
\qquad\qquad\qquad  Solve (\ref{forward1}) to obtain and store the optimal occupation measures $\hat{\gamma}_{z}$; \\
\qquad\qquad\qquad  Calculate the lower bound $F_{LB,z} $. \\
Step (c): \textbf{Backward recursion:} \\
\qquad\qquad\qquad  Solve (\ref{Back-Dual 1}) to get the dual function $\hat{\phi}_z$. \\
\qquad\qquad\qquad Calculate the upper bound $B_{UB ,z} $. \\
Step (d): If $B_{UB,z} -F_{LB ,z} < \varepsilon $, \\
\qquad\qquad\qquad Then Stop; \\
\qquad\qquad\qquad Else $z \leftarrow z+1$; go to Step (b). \\
\hline
\hline
\end{tabular}
{
\begin{remark}
\begin{enumerate}
\item From a theoretical point of view,\begin{enumerate}
\item the linear programming acts as a compactification method. In particular, under rather general lower semi-continuity (l.s.c.) assumptions on the cost, the optimal measure will exist even when the strict control $\pi^*=\pr{u^*,L^*,I^*}$ might not. This is the case without any further convexity assumptions;
\item the dual dynamic programming can be generalized to mere measurable costs (see \cite{GoreacIvascu_2013}), where the classical dynamic programming might fail to work.
\end{enumerate}
\item From an algorithmic point of view,
\begin{enumerate}
\item this offers a non-trivial tool going beyond the case $k=\infty$ to which the approach in \cite{azcue2014stochastic} can be associated;
\item because of the requirement of mere l.s.c.  for the backward candidate, this algorithm can naturally be generalized to several steps. This is essential in order to compute the optimal measure in addition to the optimal value.
\end{enumerate}
\end{enumerate}
\end{remark}
}

\section{Value Function with a Penalty for Bankruptcy}\label{Section6}
 In this section, we will consider the value function with a penalty for bankruptcy of linear type, i.e.
\[ V(x):=\sup_{\pi\in\Pi(x)}v\pr{x,\pi},\ \forall x\in\mathbb{R}  .\]
The cost functional is
\begin{equation}
 v\pr{x, \pi} = \mathbb{E}_x\pp{\int_0^{\sigma_{0-}^{\pi}}e^{-qt}\pr{dL_s-kdI_s} +e^{-q\sigma_{0-}^{\pi}} (-a +bX_{\sigma_{0-}^{\pi}}^{\pi})  } ,
\end{equation}
where $a,b\geq 0$ and $e^{-q\sigma_{0-}^{\pi}} (-a +bX_{\sigma_{0-}^{\pi}}^{\pi}) $ is the penalty. for declaring bankruptcy. We need to give one more assumption on $a,\ b$ appearing in the cost functional:
\begin{equation*}
\pp{\mathbf{A_{a,b}}}:\qquad\qquad\qquad\qquad\qquad \frac{\norm{p}_0}{\lambda} -a -b\mathbb{E} \pp{C} \geq 0. \qquad\qquad\qquad\qquad\qquad\qquad\qquad\qquad\qquad\qquad
\end{equation*}
Now, we recall the surplus process
\begin{equation*}
X_t^\pi:=x+\int_0^t p^{u_s}\pr{X_s^\pi}ds-\sum_{i=1}^{N_t}u_{\tau_i}\pr{C_i}-L_t+I_t;
\end{equation*}
and next, we will give some basic properties in this case which are kinds of different from the Proposition \ref{PropBasicPropertiesXV} given before.
\begin{proposition}
\label{estimatepenalty1}
\begin{enumerate}
\item  For any $\pi\in\Pi(x)$, the following holds true $V(x)\geq x+ \frac{\lambda}{\lambda +q} \pr{   \frac{\norm{p}_0 }{\lambda} -a -b \mathbb{E}[C]  }$.
\item If $q\geq\pp{p}_1$, then, for every $\pi\in\Pi(x)$, the associated gain satisfies $v(x,\pi)\leq x+\frac{\norm{p}_0}{q} \leq x+\frac{\norm{p}_0}{[p]_1}$.
\item For every $\varepsilon>0$ small enough,  one has $V(x+\varepsilon)-V(x)\in\pp{0,k\varepsilon}$.
\end{enumerate}
\end{proposition}
\begin{proof}
We consider the no-injection, no-reinsurance, pay all reserve (and declare bankruptcy at first claim $\tau_1$) strategy denoted by $\pi^0$:
\begin{equation*}
\begin{split}
V(x) &\geq v (x, \pi^0) = \mathbb{E}_x \pp{\int_{0}^{\tau_1} e^{-qs} p^{u_s} (0) ds + e^{-q\tau_1} \pr{-a+bX_{\tau_1}^{\pi}}} \\
&= x\!+\! \mathbb{E} \pp{\int_{0}^{\tau_1} \! e^{-qs} \norm{p}_0 ds \!+\! e^{-q\tau_1} \pr{-a+bX_{\tau_1}^{\pi^0}} \!} \!=\\ &=x\!+\frac{\norm{p}_0}{q} \!-\! \pr{\! a \!+\! \frac{\norm{p}_0}{q} \!} \mathbb{E} \pp{e^{-q\tau_1}} + b\mathbb{E} \pp{e^{-q\tau_1} X_{\tau_1}^{\pi^0}} \\
&= x+\frac{\norm{p}_0}{q} - \frac{\lambda}{\lambda +q} \pr{a+\frac{\norm{p}_0}{q}} +b \int_{0}^{\infty} \lambda e^{-\lambda t} e^{-qt} dt \int_{\mathbb{R}_+} -y F(dy) \\
& = x+ \frac{\lambda}{\lambda +q} \pr{   \frac{\norm{p}_0 }{\lambda} -a -b \mathbb{E}[C]  }.
\end{split}
\end{equation*}
\\Item 2: Since $X_{\sigma_{0-}^{\pi}}^{\pi} \leq 0$ at the time of bankruptcy, it follows that
\begin{equation*}
\begin{split}
 v(x,\pi) &\leq \mathbb{E}_x\pp{\int_0^{\sigma_{0-}^{\pi}} e^{-qs}\pr{dL_s-kdI_s} -ae^{-q\sigma_{0-}^{\pi}}  } \\
& \leq \mathbb{E}_x \pp{\int_0^{\sigma_{0-}^{\pi}}e^{-qs}\pr{dL_s-kdI_s}   }  \leq x+ \frac{\norm{p}_0}{q} \leq x+\frac{\norm{p}_0}{[p]_1}. \\
\end{split}
\end{equation*}
\\Item 3:
\begin{itemize}
\item If $\pr{u,L,I}=:\pi\in\Pi (x)$, take $\tilde{\pi}=\pr{u,L+\varepsilon ,I}$ (i.e. pay $\varepsilon$ at time 0, from the reserve $x+\varepsilon$ and then follow the policy $\pi$) and $\tilde{\pi} \in \Pi \pr{x+\varepsilon}$. Hence, $v (x+\varepsilon ,\tilde{\pi}) = \varepsilon +v (x,\pi)$, $V(x+\varepsilon) \geq \varepsilon +v(x,\pi)$ and $V(x)+\varepsilon \leq V(x+\varepsilon)$.
\item If $\pr{u,L,I}=:\pi \in\Pi (x+\varepsilon)$, take $\tilde{\pi} =\pr{u,L,I+\varepsilon}$ (i.e. inject $\varepsilon$ at time 0, from the reserve x, then follow $\pi$) and it is easy to see that $\tilde{\pi} \in \Pi (x)$. Thus, $v(x+\varepsilon ,\pi) = v(x,\tilde{\pi}) +k\varepsilon$, $v (x+\varepsilon ,\pi) \leq V(x)+ k\varepsilon$ and finally $V(x+\varepsilon) \leq V(x)+k\varepsilon$.
\end{itemize}
\end{proof}
\begin{proposition}\label{PropEstimILXpenalty}
Let $x\geq 0$ and $\pi:=\pr{u,L,I}\in\Pi(x)$ be $1$-optimal (in the sense that the associated cost $v(x,\pi)\geq V(x)-1$). Then, for every $t\geq 0$,
\begin{enumerate}
\item the injection process satisfies $\mathbb{E}_x\pp{\displaystyle{\int_{\pp{0,t\wedge\sigma_{0-}^{\pi}}}} e^{-qs}dI_s} \leq \frac{2\norm{p}_0 +[p]_1}{(k-1) [p]_1} + \frac{\lambda \pr{a+b\mathbb{E} [C]}}{(k-1) \pr{\lambda +[p]_1}} ;$
\item the dividend process satisfies $\mathbb{E}_x\pp{\displaystyle{\int_{\pp{0,t\wedge\sigma_{0-}^{\pi}}}}e^{-qs}dL_s} \leq x+\frac{(k+1) \norm{p}_0 +[p]_1}{(k-1) [p]_1} + \frac{\lambda \pr{a+b\mathbb{E} [C]}}{(k-1) \pr{\lambda +[p]_1}} ;$
\item the associated solution satisfies \[\mathbb{E}_x\pp{\underset{t\geq 0}\sup\ e^{-q\pr{t\wedge\sigma_{0-}^\pi}}\pr{X^\pi_{t\wedge\sigma^{\pi}_{0-}}\vee 0}}\leq  x+\frac{(k+1) \norm{p}_0 +[p]_1}{(k-1) [p]_1} + \frac{\lambda \pr{a+b\mathbb{E} [C]}}{(k-1) \pr{\lambda +[p]_1}} ;\]
\item Furthermore, $\mathbb{E}_x\pp{\displaystyle{\int_{\pp{0,\sigma_{0-}^\pi}}} e^{-qt}\pr{X^\pi_{t}\vee 0}dt}\leq \frac{1}{q-\pp{p}_1} \pr{x+\frac{(k+1) \norm{p}_0 +[p]_1}{(k-1) [p]_1} + \frac{\lambda \pr{a+b\mathbb{E} [C]}}{(k-1) \pr{\lambda +[p]_1}}}.$
\end{enumerate}
\end{proposition}
\begin{proof}
Under the assumptions introduced before and using the previous proposition, one has
\begin{equation*}
\begin{split}
& x+ \frac{\lambda}{\lambda +q} \pr{   \frac{\norm{p}_0 }{\lambda} -a -b \mathbb{E}[C]  } -1 \leq V(x)-1 \leq v(x,\pi) \\
\leq & \mathbb{E}_x \!\pp{\int_{\pp{0, t\wedge \sigma_{0-}^{\pi}}} e^{-qs} \pr{dL_s -kdI_s}  +e^{-q \pr{t\wedge\sigma_{0-}^{\pi}}} V\pr{X_{t\wedge\sigma_{0-}^{\pi}}^{\pi}}} \\
\leq & \mathbb{E}_x\pp{\int_{\pp{0,t\wedge\sigma_{0-}^{\pi}}} \!e^{-qs}\!\pr{dL_s \!-\!dI_s}\!+\!e^{-q\pr{t\wedge\sigma_{0-}^{x,\pi}}} \!\pr{\! \max\set{X_{t\wedge\sigma_{0-}^{\pi}}^{\pi},0}\!+\!\frac{\norm{p}_0}{[p]_1}} \!}\\&-\!(k\!-\!1)\mathbb{E}_x \!\pp{\int_{\pp{0,t\wedge\sigma_{0-}^{\pi}}} \!e^{-qs}dI_s}\\
\leq & x+ \frac{2\norm{p}_0}{[p]_1} -(k-1)\mathbb{E}_x\pp{\int_{\pp{0,t\wedge\sigma_{0-}^{\pi}}}e^{-qs}dI_s}. \\
\end{split}
\end{equation*}
It follows that
\begin{equation}
\begin{split}
\mathbb{E}_x\pp{\int_{\pp{0,t\wedge\sigma_{0-}^{\pi}}}e^{-qs}dI_s} &\leq \frac{1}{k-1} \set{\frac{2\norm{p}_0}{[p]_1} +\frac{\lambda}{\lambda +q} \pr{  a +b \mathbb{E}[C]  } +1} \\
 &\leq \frac{2\norm{p}_0 +[p]_1}{(k-1) [p]_1} + \frac{\lambda \pr{a+b\mathbb{E} [C]}}{(k-1) \pr{\lambda +[p]_1}}.\\
\end{split}
\end{equation}
By using item 3 of the last proposition and the previous estimates, we can deduce that
\begin{equation}
\begin{split}
&\mathbb{E}_x\pp{\sup_{t\geq 0}e^{-qt}\pr{X^\pi_{t\wedge\sigma^{x,\pi}_{0-}}\vee 0}}+\mathbb{E}_x\pp{\int_{\pp{0,t\wedge\sigma_{0-}^{x,\pi}}}e^{-qs}dL_s} \\
\leq & x+\frac{\norm{p}_0}{[p]_1} +\mathbb{E}_x \pp{\int_{\pp{0,t\wedge\sigma_{0-}^{\pi}}} e^{-qs} dI_s} \\
\leq & x+\frac{(k+1) \norm{p}_0 +[p]_1}{(k-1) [p]_1} + \frac{\lambda \pr{a+b\mathbb{E} [C]}}{(k-1) \pr{\lambda +[p]_1}}. \\
\end{split}
\end{equation}
Finally, by using a similar method to the one in Proposition \ref{PropEstimILX}, it follows that
\begin{equation}
\begin{split}
\mathbb{E}_x\pp{\int_{\pp{0,\sigma_{0-}^\pi}} e^{-qt}\pr{X^\pi_{t}\vee 0}dt}  &\leq  \frac{1}{q-[p]_1} \pr{x +\frac{\norm{p}_0}{[p]_1} +\mathbb{E}_x \pp{\int_{0}^{\sigma_{0-}^{\pi}} e^{-qs} dI_s }    } \\
&\leq \frac{1}{q-\pp{p}_1} \pr{x+\frac{(k+1) \norm{p}_0 +[p]_1}{(k-1) [p]_1} + \frac{\lambda \pr{a+b\mathbb{E} [C]}}{(k-1) \pr{\lambda +[p]_1}}}.
\end{split}
\end{equation}
\end{proof}

The Hamilton-Jacobi-Bellman variational inequality in this case is the same as before.
Under the additional assumption on $a,\ b$, we have the following proposition.
\begin{proposition}\label{VsupersolP}
The function $V$ is a non-negative absolutely continuous ($\mathcal{AC}$) super-solution of \eqref{HJB} such that $V(x)\leq x+\frac{\norm{p}_0}{[p]_1}$, for all $x\geq 0$, and $V(0)\geq \frac{\lambda}{\lambda +q} \pr{\frac{\norm{p}_0}{\lambda} -a -b \mathbb{E} [C]}>0$.
\end{proposition}
The proof is very similar to the one without penalty and will, therefore, be omitted.
\section*{Conclusions}
This paper provides the theoretical basis for dual algorithms in connection with complex random systems with mixed control interventions. The numerical implementation of these methods using classical and deep networks methods make the object of on-going research. \\

The explicit $a,b$-barrier strategies explained in \cite{avram2020equity} and \cite{AGAS2022} should provide a benchmark in the case without reinsurance. The limit case $k=\infty$ is morally reduced to the better-known framework with no-injection and provides a second benchmark (using \cite{azcue2014stochastic}). Finally, the case of proportional reinsurance will provide a further benchmark (with three explicit parameters in well-chosen cases).\\

The method is expected to be intensively used for such reinsurance problems when the candidates to optimality are either difficult to guess or the verification result is simply too fastidious.

\section{Appendix}\label{SectionA}
\subsection{Proof of Proposition \ref{PropEstimILX}}
\begin{proof}[Proof of Proposition \ref{PropEstimILX}]
Under the assumption in our statement, and owing to Proposition \ref{PropBasicPropertiesXV} (assertions 3. and 5.),
\begin{equation}
\label{EstimI}
\begin{split}
&V(x)-1\leq v(x,\pi)\\
\leq &\mathbb{E}_x\pp{\int_{\pp{0,t\wedge\sigma_{0-}^{\pi}}}e^{-qs}\pr{dL_s-kdI_s}+e^{-q\pr{t\wedge\sigma_{0-}^{\pi}}}V\pr{X_{t\wedge\sigma_{0-}^{\pi}}^{\pi}}}\\
\leq &\mathbb{E}_x\pp{\int_{\pp{0,t\wedge\sigma_{0-}^{\pi}}}e^{-qs}\pr{dL_s-dI_s}+e^{-q\pr{t\wedge\sigma_{0-}^{\pi}}}\pr{\max\set{X_{t\wedge\sigma_{0-}^{\pi}}^{\pi},0}+\frac{\norm{p}_0}{\pp{p}_1}}}\\&-(k-1)\mathbb{E}_x\pp{\int_{\pp{0,t\wedge\sigma_{0-}^{\pi}}}e^{-qs}dI_s}\\
\leq &x+\frac{2\norm{p}_0}{\pp{p}_1}-(k-1)\mathbb{E}_x\pp{\int_{\pp{0,t\wedge\sigma_{0-}^{\pi}}}e^{-qs}dI_s}.
\end{split}
\end{equation}
It follows that $\displaystyle{\mathbb{E}_x\pp{\int_{\pp{0,t\wedge\sigma_{0-}^{\pi}}}e^{-qs}dI_s}\leq \frac{2\norm{p}_0+\pp{p}_1}{(k-1)\pp{p}_1}}$. Using Proposition \ref{PropBasicPropertiesXV} assertion 3. and the previous estimate, it follows that \[\mathbb{E}_x\pp{\sup_{t\geq 0}e^{-qt}\pr{X^\pi_{t\wedge\sigma^{\pi}_{0-}}\vee 0}}+\mathbb{E}_x\pp{\int_{\pp{0,t\wedge\sigma_{0-}^{\pi}}}e^{-qs}dL_s}\leq x+\frac{(k+1)\norm{p}_0+\pp{p}_1}{(k-1)\pp{p}_1},\]thus providing the assertions 2. and 3.\\
Finally, owing to Proposition \ref{PropBasicPropertiesXV} (assertion 3. written for $\pp{p}_1$ instead of $q$), it follows that, for $0\leq t\leq \sigma_{0-}^\pi$,
\[X_t^\pi\vee0\leq xe^{\pp{p}_1t}+\norm{p}_0\frac{e^{\pp{p}_1t}-1}{\pp{p}_1}+\int_{\pp{0,t}}e^{\pp{p}_1\pr{t-s}}dI_s.\]
One multiplies this equation by $e^{-qt}$, integrates with respect to $dt$ on $\pp{0,\sigma_{0-}^\pi}$ and uses Fubini arguments to deal with the $dI_sdt$ part to infer\[\int_{\pp{0,\sigma_{0-}^\pi}}e^{-qt}\pr{X_t^\pi\vee0}dt\leq x \frac{1}{q-\pp{p}_1}+\frac{\norm{p}_0}{\pp{p}_1}\pr{ \frac{1}{q-\pp{p}_1}-\frac{1}{q}}+ \frac{1}{q-\pp{p}_1}\int_{\pp{0,\sigma_{0-}^\pi}}e^{-qs}dI_s.\]The fourth assertion is complete by taking the expectancy under $\mathbb{P}_x$ and using the first assertion.
\end{proof}

\subsection{(Sketch of the) Proof of the Super-Solution Property }\label{Sub7.1}
 Before proceeding with the elements of proof, let us give a couple of comments on the dynamic programming principle.
\begin{proposition}\label{DPP}
For every initial $x\in\mathbb{R}_+$ and every $t\geq 0$, the following equality holds true\[V(x)=\sup_{\pi\in\Pi(x)}\mathbb{E}_x\pp{\int_{\pp{0,t\wedge\sigma_{0-}^{\pi}}}e^{-qs}\pr{dL_s-kdI_s}+e^{-q\pr{t\wedge\sigma_{0-}^\pi}}V\pr{X_{t\wedge\sigma_{0-}^\pi}^{\pi}}}.\]
\end{proposition}
\begin{proof}[Elements of Proof of Proposition \ref{DPP}]
The proof is quite standard, inspired by \cite{azcue2005optimal} and the only difficulty resides in presence of possible unbounded injection $I$.

We only need to prove that $V(x)$ cannot be lower than the right-hand expression (the converse inequality being straightforward) and this right-hand term is denoted by $\tilde{V}(x)$. As usual, the arguments rely on the uniform continuity of $V$ (following from Proposition \ref{PropBasicPropertiesXV}).

Let us fix $x\in\mathbb{R}_+$ and $\varepsilon>0$. We begin with noting that $\tilde{V}(x)\geq x$. The argument requires considering  the strategy $\pi^0$ consisting in no-reinsurance, no capital injection, $L_{0}=x$ followed by distributing all premiums as dividends prior to the first claim at which bankruptcy is declared. Secondly, reasoning as we have already done in Proposition \ref{PropEstimILX}, it follows that as soon as $\pi$ is $\varepsilon$-optimal (for $\tilde{V}(x)$), one has
\[\mathbb{E}_x\pp{\int_{\pp{0,t\wedge\sigma_{0-}^{\pi}}}e^{-qs}dI_s}\leq \frac{2\norm{p}_0+\pp{p}_1}{(k-1)\pp{p}_1}.\] Using Proposition \ref{PropBasicPropertiesXV} assertion 3., it follows that \[\mathbb{E}_x\pp{X^\pi_{t\wedge\sigma^{\pi}_{0-}}\vee 0}\leq e^{qt}\pr{x+c},\]where $c>0$ is a generic constant (independent of $x$ and $\pi$) that is allowed to change from one line to another. In particular, $\mathbb{P}_x\pr{X^\pi_{t\wedge\sigma^{\pi}_{0-}}>\frac{e^{qt}\pr{x+c}}{\varepsilon}}\leq \varepsilon$ and define $\mathbb{K}:=\set{y\in\mathbb{R}_+:\ y\leq \frac{e^{qt}\pr{x+c}}{\varepsilon}}$ to get $\mathbb{P}_x\pr{\set{X^\pi_{t\wedge\sigma^{\pi}_{0-}}\vee 0  \notin \mathbb{K}}}\leq \varepsilon.$
Since $V$ is uniformly continuous, there exists a finite family $\pr{x^i}_{1\leq i\leq N}$ covering $\mathbb{K}$ such that, for every $x\in \left[x^i,x^{i+1} \right)$, one has \[\max\set{V(x)-V\pr{x^i},x-x^i}\leq \varepsilon.\]
\hspace{1.5em}For every such $x^i$, there exists a strategy $\pi^i\in\Pi\pr{x^i}$ such that $v\pr{x^i,\pi^i}\geq V\pr{x^i}-\varepsilon.$ We define $\pi^*$ by setting
\begin{enumerate}
\item $\pi^*_s=\pi_s$, on \[\set{\sigma_{0-}^{\pi}\leq t,\ s\geq 0}\cup\set{\sigma_{0-}^{\pi}>t> s};\]
\item $\pi^*_s=\pi^i_{s-t}+\pr{0,L_t+X_t^\pi-x^i,I_t}$ on $\set{\sigma_{0-}^{\pi}>t,\ X_t^\pi\in\left[x^i,x^{i+1}\right),\ s\geq t}$ (i.e. pay the surplus $X_t^\pi-x^i$ at time $t$, then follow the strategy $\pi^i$ from the (new) reserve $x^i$);
\item $\pi^*_s=\pr{0,L_t+X_t^\pi,I_t}$ on $\set{\sigma_{0-}^{\pi}>t, X_t^\pi\notin  \mathbb{K},\ s\geq t}$ (i.e. give all the surplus $X_t^\pi$ in dividends then wait for ruin).
\end{enumerate}Then, the strategy $\pi^*$ is admissible and
\begin{enumerate}
\item On $\set{\sigma_{0-}^{\pi}\leq t}$, $X_{\sigma_{0-}^{\pi}\wedge t}< 0,\ \mathbb{P}-a.s.$ and everything is contained in the running cost.
\item On $\set{\sigma_{0-}^{\pi}>t,\ X_t^\pi\in\left[x^i,x^{i+1}\right)}$, using Proposition \ref{PropBasicPropertiesXV}, assertion 6, one has \begin{equation}\label{DPP1}\footnote{With a proper shift notation, using a strong-Markov family, the left-hand term corresponds to $v\pr{X_t^\pi,\pi^*}$ on the $\mathcal{F}_{t+}$-measurable set $\set{\sigma_{0-}^{\pi}>t,\ X_t^\pi\in\left[x^i,x^{i+1}\right)}$.} X_t^\pi-x^i+v\pr{x^i,\pi^i}\geq V\pr{x^i}-\varepsilon\geq V\pr{X_t^\pi}-\pr{k+1}\varepsilon.
\end{equation}\item On the set $\set{\sigma_{0-}^{\pi}>t}$, one has $X_t^\pi\geq 0$. As a consequence, \begin{equation}
\label{DDP2}
\begin{split}
&\mathbb{E}_x\pp{\mathbf{1}_{\sigma_{0-}^{\pi}>t,\ X_t^\pi \notin \mathbb{K}}e^{-qt}V\pr{X_t^\pi}}-\mathbb{E}_x\pp{\mathbf{1}_{\sigma_{0-}^{\pi}>t,\ X_t^\pi \notin \mathbb{K}}e^{-qt}X_t^\pi}\\\leq &e^{-qt}\frac{\norm{p}_0}{\pp{p}_1}\mathbb{P}_x\pr{{\sigma_{0-}^{\pi}>t,\ X_t^\pi \notin \mathbb{K}}}\leq \frac{\norm{p}_0\varepsilon}{\pp{p}_1}.
\end{split}
\end{equation}
\end{enumerate}
Putting all these together, it follows that
\begin{equation}
\begin{split}
\tilde{V}(x)-\varepsilon\leq &\mathbb{E}_x\pp{\int_{\pp{0,t\wedge\sigma_{0-}^{\pi}}}e^{-qs}\pr{dL_s-kdI_s}+e^{-q\pr{t\wedge\sigma_{0-}^{\pi}}}V\pr{X_{t\wedge\sigma_{0-}^{\pi}}^{\pi}}}\\
&\leq v\pr{x,\pi^*} \!+\! (k+1)\varepsilon\mathbb{P}_x\pr{\sigma_{0-}^{\pi}>t,X_t^\pi\in \mathbb{K}} \!+\! \frac{\norm{p}_0\varepsilon}{\pp{p}_1} \!\leq \! V(x)+\pr{k+1+\frac{\norm{p}_0}{\pp{p}_1}}\varepsilon.
\end{split}
\end{equation}
Our proof is now complete by recalling that $\varepsilon>0$ is arbitrary.
\end{proof}
\begin{remark}
Of course, it is straight-forward to generalize this from deterministic times $t$ to an exogenous stopping time $t\wedge \tau_1$ (note that $\tau_1\leq \sigma_{0-}^\pi,\ \mathbb{P}-a.s.$) showing the behavior at $t$ or the first claim, i.e. \[\sup_{\pi\in\Pi(x)}\mathbb{E}_x\pp{\int_{\pp{0,t\wedge\tau_1}}e^{-qs}\pr{dL_s-kdI_s}+e^{-q\pr{t\wedge\tau_1}}V\pr{X_{t\wedge\tau_1}^{\pi}}}=V(x).\]
\end{remark}
We proceed with the proof of Proposition \ref{Vsupersol} on the non-negative axis.\\
\begin{proof}[(Sketch of the) Proof of Proposition \ref{Vsupersol}]
Let $u\in\mathcal{R}$ be fixed and let us consider $\pi=\pr{u,0,0}$ (no-dividend, no injection policy). For convenience, we let $\Phi_t^{x,u}$ denote the non-decreasing solution of $d\Phi_t^{x,u}=p^u\pr{\Phi_t^{x,u}}dt,\ \Phi_0^{x,u}=x$. One easily notes that $\Phi_t^{x,u}\geq x+tp^u(x),$ for all $t\geq 0$.\\
Taking into account the dynamic programming principle (see the previous remark) and the monotonicity of $V$, one has
\begin{align*}V(x)\geq &\mathbb{E}_x\pp{e^{-q\pr{t\wedge\tau_1}}V\pr{X_{t\wedge\tau_1}^{\pi}}}\\\geq&\mathbb{P}_x\pr{\tau_1>t}e^{-qt}V\pr{x+p^u(x)t}+\mathbb{E}_x\pp{\mathbf{1}_{\tau_1\leq t}e^{-q\tau_1}\int_{\mathbb{R}_+}V(x-y)F^u(dy)}.\end{align*}One recalls that $V$ is $\mathcal{AC}$, $\mathbb{P}_x\pr{\tau_1>t}=e^{-\lambda t}$ and $\mathbb{E}_x\pp{\mathbf{1}_{\tau_1\leq t}e^{-q\tau_1}}=\frac{\lambda\pr{1-e^{-(q+\lambda)t}}}{\lambda+q}$, subtracts $V(x)$ and divides the resulting equation by $t$ and finally takes the limit as $t\rightarrow 0+$ to infer \[\mathcal{L}^uV(x)=-(\lambda+q)V(x)+p^u(x)V'(x)+\lambda\int_{\mathbb{R}_+}V(x-y)F^u(dy)\leq 0,\]at every Lebesgue point of $V$, hence almost surely on $\mathbb{R}_+$.
\end{proof}
\subsection{Proof of Proposition \ref{PropSupersol}}
\begin{proof}[Proof of Proposition \ref{PropSupersol}]
Let us fix a function $\rho$ supported on $(0,1)$ that is continuously differentiable and non negative.  We construct, in a standard way,  a sequence of mollifiers by setting, for $n\geq 1$, $\rho_n(x):=n\rho\pr{nx}$, and \[\psi_n(x):=\int_{\pr{0,\frac{1}{n}}}\phi(x+s)\rho_n(s)ds,\ \forall x\in\mathbb{R}.\] This sequence $\psi_n$ consists of class $C^1$ functions. Since $\phi$ is absolutely continuous, $1\leq \phi'\leq k, a.e.$ on $\left[x_0,\infty\right)$. The same applies to $\psi'_n$. These functions comply with properties (b) and (c) and, owing to the monotonicity of $\phi$, $\displaystyle{\phi(x)\leq \psi_n(x)\leq \phi(x)+\int_{\pr{0,\frac{1}{n}}}nks\rho(ns)ds\leq \phi(x)+\frac{k}{n}}$.
Let us set, for $x\geq 0$, \[\delta_n:=\sup\set{\phi'(y): y\in\pp{x,x+\frac{1}{n}}\cap\mathcal{D}(\phi)},\]where $\mathcal{D}(\phi)$ stands for the set on which $\phi$ is differentiable. By construction, $\psi_n'(x)\leq \delta_{\frac{n}{2}}$. Owing to $\delta_m\in\pp{1,k},\ \forall m\geq 1$, there exists some point $x_n\in\pp{x,x+\frac{2}{n}}\cap\mathcal{D}(\phi)$ such that $\phi'\pr{x_n}\geq \delta_{n/2}-\frac{1}{n\bar{p}(x)}$. For every $u\in\mathcal{R}$, using the mononicity of $\phi,\psi_n$ and $p^u$ as well as the choice of $x_n\in\pp{x,x+\frac{2}{n}}$ and the super-solution condition for $\phi$ at $x_n$, we get
\begin{equation}\label{EstimLu}\
\begin{split}
&\mathcal{L}^u\psi_n(x)\\
\leq &\mathcal{L}^u\phi\pr{x_n}+p^u(x)\psi_n'(x)-p^u\pr{x_n}\phi'\pr{x_n}\\
&+\lambda\pr{\int_0^\infty\psi_n\pr{x-y}dF^u(y)-\int_0^{\infty}\phi\pr{x_n-y}dF^u(y)}+\pr{\lambda+q}\pr{\phi\pr{x_n}-\psi_n\pr{x}}\\
\leq &p^u(x)\!\pr{\psi_n'(x) \!-\! \phi'\pr{x_n}} \!+ \! \lambda\! \pr{\int_0^\infty \! \pr{\psi_n\pr{x\!-\! y} \!-\! \phi\pr{x\!-\! y}}dF^u(y) \! } \!+\!\pr{\lambda \!+\!q}\!\pr{\frac{2k}{n}+\phi\pr{x}\!-\!\psi_n(x) \!}\\
\leq &\frac{p^u(x)}{n\bar{p}(x)  }  + \lambda \frac{k}{n}+\pr{\lambda+q}\frac{2k}{n}\leq \frac{1+k\pr{3\lambda+2q}}{n}.
\end{split}
\end{equation}
\begin{enumerate}
\item By setting $\phi_n(y):=\psi_n(y)+\frac{1+(3\lambda+2q)k}{qn}, y\in\mathbb{R}$, it follows that
$\phi(y)\leq \psi_n(y)\leq \phi_n(y)\leq \phi(y)+\frac{k+\frac{1+(3\lambda+2q)k}{q}}{n}$ and $\phi_n'(y)=\psi_n'(y)$. Furthermore, using \eqref{EstimLu}, we conclude that\begin{align*}\mathcal{L}^u\phi_n(y)\leq&\mathcal{L}^u\psi_n(y)+\lambda\int_0^\infty \pr{\phi_n(y-z)-\psi_n(y-z)}dF^u(z)-(\lambda+q)\pr{\phi_n(y)-\psi_n(y)}\\\leq &\mathcal{L}^u\psi_n(y)-q\frac{1+(3\lambda+2q)k}{qn}\leq 0,\textnormal{ for }y\geq 0.\end{align*}
\item To prove the second assertion, we begin with fixing $x\geq 0$. Owing to It\^{o}'s formula applied to $\psi_n$ (cf. the first part of the proof), we get
\begin{equation}\label{AC_estim1}\begin{split}&\mathbb{E}_x\pp{e^{-q\pr{t\wedge\sigma_{0-}^{\pi}}}\psi_n\pr{X_{t\wedge\sigma_{0-}^{\pi}}^\pi}}-\psi_n(x)\\\leq
&\mathbb{E}_x\pp{\int_{\pp{0,t\wedge\sigma_{0-}^{\pi}}}\set{e^{-qs}\pr{kdI_s-dL_s}+e^{-qs}\mathcal{L}^{u_s}\psi_n\pr{X_{s-}^\pi}ds}}
\end{split}\end{equation}
When $\pi=\pr{u,L,I}$ is such that $v(x,\pi)+1\geq V(x)$, then, by Proposition \ref{PropEstimILX}, one has
\[\mathbb{E}_x\pp{\int_{\pp{0,\sigma_{0-}^{\pi}}}e^{-qs}dI_s}\leq \frac{2\norm{p}_0+\pp{p}_1}{(k-1)\pp{p}_1}.\]
We aim at applying Lebesgue's dominated convergence arguments (as the time $t$ or $n$ go to $\infty$). To be able to do this, one uses the later inequality, relying on Proposition \ref{PropBasicPropertiesXV} assertion 3., together with $0\leq \psi_n(y)\leq y+c+\frac{k}{n}$ and the inequality\[-\pr{\lambda+q}\psi_n(y)\leq\mathcal{L}^u\psi_n(y)\leq k\pr{\norm{p}_0+\pp{p}_1y}+\lambda\psi_n(y),\ \forall y\geq 0.\]The same applies to $\phi=\psi_\infty$.
Allowing $t$ to go to $\infty$ in \eqref{AC_estim1} and since $\psi_n$ is non negative, we have
\begin{align}\label{AC_estim2}\mathbb{E}_x\pp{\int_{\pp{0,\sigma_{0-}^{\pi}}}e^{-qs}\pr{-kdI_s+dL_s}}\leq \psi_n(x)+
\mathbb{E}_x\pp{\int_{\pp{0,\sigma_{0-}^{\pi}}}e^{-qs}\mathcal{L}^{u_s}\psi_n\pr{X_{s-}^\pi}ds}.
\end{align}
We conclude, by passing $n\rightarrow\infty$ in \eqref{AC_estim2}, invoking \eqref{EstimLu} and by maximizing over $\pi\in\Pi(x)$ that the inequality $V(x)\leq \phi(x)$ holds true.
\end{enumerate}
\end{proof}
\subsection{Proofs of Theorem \ref{ThmDual}, Corollary \ref{Cor} and Lemma \ref{infchange}}
We gather here these results for which the arguing is somewhat similar.\\

\begin{proof}[Proof of Theorem \ref{ThmDual}]
We fix $x\geq 0$. The inclusion of the occupation measures in $\Theta(x)$ (as presented before) guarantees that $V(x)\leq \Lambda(x)$.

Furthermore, we let $\phi\in C^1\pr{\mathbb{R};\mathbb{R}_+}$, with linear-growth, derivative $\phi'\in\pp{1,k}$ on $\mathbb{R}_+$ and satisfying $\mathcal{L}^u\phi(y)\leq 0$, for all $u\in\mathcal{R}$ and all $y\in\mathbb{R}_+$.

For $\gamma\in \Theta(x)$, the constraint condition written for the test function $\phi$, combined with the non-positiveness of $\mathcal{L}^u\phi\pr{y_2}$ for all $y_2\geq 0$, hence $\gamma_1$-a.s., and the restrictions on the derivative of $\phi$ yield \begin{align*}0\leq&\int_{\mathbb{R}_+\times\mathbb{R}} e^{-qs_1}\phi\pr{y_1}\gamma_0\pr{ds_1,dy_1}\leq\phi(x)+\int_{\mathbb{R}_+} \phi'\pr{y_2}\pr{\gamma_3-\gamma_2}\pr{\mathbb{R}_+,\mathbb{R},\mathbb{R}_+,dy_2,\mathcal{R},\mathbb{R}_+}\\&\leq \phi(x)+\pr{k\gamma_3-\gamma_2}\pr{\mathbb{R}_+,\mathbb{R},\mathbb{R}_+,\mathbb{R}_+,\mathcal{R},\mathbb{R}_+}.\end{align*}
Then, by taking infimum over such $\phi$ respectively the supremum over $\gamma\in\Theta(x)$ , one gets $\Lambda(x)\leq\Lambda^*(x)$.

To conclude, we need to prove that $\Lambda^*(x)\leq V(x)$. To achieve this, one relies on the auxiliary constructions in Proposition \ref{PropSupersol} and on the characterization of $V$ as super-solution in Proposition \ref{Vsupersol}. Indeed, for every $n\geq 1$, the functions $\phi_n$ constructed in Proposition \ref{PropSupersol} are of linear growth (cf. (a) and the properties of $V$), with derivative belonging to $\pp{1,k}$ on some interval containing $\mathbb{R}_+$ (cf. (b)) and with a non-positive generator (cf. (d)). It follows that $\Lambda^*(x)\leq\phi_n(x)\leq V(x)+\frac{\tilde{c}}{n}$. Our proof of \eqref{Dual} is complete by recalling that $\tilde{c}$ is independent of $n$ and $n$ is arbitrarily large.\\
To prove \eqref{Dualk}, one notes that $\Lambda^*$ cannot exceed the right-hand member of \eqref{Dualk} and Remark \ref{dPsi} allows to prove (as we did just before) that the right-hand member of \eqref{Dualk} does not exceed $V(x)$. The proof is now complete.
\end{proof}\\

Let us now turn to the proof of the Corollary.\\
\begin{proof}[Proof of Corollary \ref{Cor}]
We recall (see Proposition \ref{DPP}), that
\[V(x)=\sup_{\pi\in\Pi(x)}\mathbb{E}_x\pp{\int_{\pp{0,t\wedge\sigma_{0-}^{\pi}}}e^{-qs}\pr{dL_s-kdI_s}+e^{-q\pr{t\wedge\sigma_{0-}^\pi}}V\pr{X_{t\wedge\sigma_{0-}^\pi}^{\pi}}}.\]
The occupation measure $\gamma$ associated to the random time $t\wedge\sigma_{0-}^{\pi}$ and to the policy $\pi$ belongs to $\Theta_t(x)$. It follows that \begin{align*}V(x)\leq &\sup_{\gamma\in\Theta_t(x)}\set{\int \pr{d\gamma_2-kd\gamma_3}+\int e^{-qs_1}V\pr{y_1}d\gamma_0}\\= &\sup_{\gamma\in\Theta_t(x)}\set{\int \pr{d\gamma_2-kd\gamma_3}+\int e^{-qs_1}{\Lambda}^*\pr{y_1}d\gamma_0}.\end{align*}

By definition, $\Lambda^*(y_1)\leq \phi\pr{y_1}$, for all $y_1\geq 0$ and $\phi$ satisfying the restrictions in $\mathbf{F}$. Moreover, since such $\phi$ have $k$-upper bounded derivative and are non-negative, $\Lambda^*\pr{y_1}=\pr{\Lambda^*(0)+ky_1}^+\leq \phi\pr{y_1}$, for all $y_1\leq 0$. As a consequence, the second term in \eqref{TwoStage} does not exceed the third one.

It follows that, in order to conclude, one should prove that
\begin{equation}\label{EstimLinDPP}
\sup_{\gamma\in\Theta_t(x)}\set{\int \pr{d\gamma_2-kd\gamma_3}+ \inf_{\phi\in \mathbf{F}}\  \int e^{-qs_1}\phi(y_1)d\gamma_0}\leq V(x).
\end{equation}

We recall that, whenever $\phi\in C^1\pr{\mathbb{R};\mathbb{R}_+}$ s.t.  $\phi'\in\pp{1,k}\textnormal{ on }\pp{0,\infty}$ and
$\mathcal{L}^u\phi(y)\leq 0$, for all $u\in\mathcal{R}$, and all $y\in\mathbb{R}_+$, by the definition of $\Theta(x)$ ,\[\int e^{-qs_1}\phi\pr{y_1}d\gamma_0=\phi(x)+\int\mathcal{L}^u\phi\pr{y_2}d\gamma_1+\int\phi'\pr{y_2}\pr{d\gamma_3-d\gamma_2}\leq \phi(x)+\int \pr{kd\gamma_3-d\gamma_2}.\]

As a consequence, it follows that, for such functions $\phi$, \[\int \pr{d\gamma_2-kd\gamma_3}+ \int e^{-qs_1}\phi(y_1)d\gamma_0\leq \phi(x).\]The claim in \eqref{EstimLinDPP} is got, as in Theorem \ref{ThmDual}, by employing the functions constructed in the second assertion of Proposition \ref{PropSupersol} and owing to Remark \ref{dPsi}. This completes our proof.
\end{proof}\\

To end this subsection, we provide a similar proof for Lemma \ref{infchange}.\\

\begin{proof}[Proof of Lemma \ref{infchange}]
For every test function $\phi$ that satisfies the constraints, the following inequality holds true,
\begin{equation*}
\int e^{-qs_1} \Lambda^*(y_1) d\gamma_{0,z} = \int e^{-qs_1} \underset{\psi \in\mathbf{F}}{\inf}\ \psi (y_1) d\gamma_{0,z}
 \leq \int e^{-qs_1} \phi (y_1) d\gamma_{0,z}.
\end{equation*}
Thus, one can deduce $ \int e^{-qs_1} \Lambda^*(y_1) d\gamma_{0,z} \leq \underset{\phi \in\mathbf{F}}{\inf}\  \int e^{-qs_1}\phi(y_1)d\gamma_{0,z} $.\\
For the converse inequality, one relies (again) on Proposition \ref{PropSupersol} and on the equality $V=\Lambda^*$ provided in Theorem \ref{ThmDual}. The functions $\phi_n$ exhibited in Proposition \ref{PropSupersol} satisfy (cf. item 1. (a)) $\phi_n(y)\leq \Lambda^*(y)+\frac{\tilde c}{n}$ (where $\tilde c$ is a constant independent of $y\in \mathbb{R}$ and of $n\geq 1$).  Furthermore, $\phi_n\in \mathbf{F}$ (again due to Proposition \ref{PropSupersol} with the help of Remark \ref{dPsi}).  One recalls that $\gamma_{0,z}\in \mathcal{P}\pr{\mathbb{R}_+\times\mathbb{R}}$ such that \[\inf_{\phi\in\mathbf{F}}\int e^{-qs_1}\phi(y_1)d\gamma_{0,z}\leq \int e^{-qs_1}\phi_n(y_1)d\gamma_{0,z}\leq \int e^{-qs_1}\Lambda^*(y_1)d\gamma_{0,z}+\frac{\tilde c}{n}.\]One gets the desired inequality by letting $n\rightarrow\infty$.
\end{proof}
\subsection{Proof of Proposition \ref{PropNegx}}

\begin{proof}[Proof of Proposition \ref{PropNegx}]
For the first assertion, one just needs to check (\ref{TwoStage}) is valid when $x<0$.

We recall that, for $x\in\mathbb{R}_-$, $V(x)=\max\set{V(0)+kx,0}$.  On the other hand, for $x<-\frac{V(0)}{k}$, $V(x)=0$, and by the definition in Eq. \eqref{ThetaExt1}, if $\gamma\in\Theta_t(x)$,
\begin{equation*}
\int \pr{d\gamma_2-kd\gamma_3}+\int e^{-qs_1}{\Lambda}^*\pr{y_1}d\gamma_0 = {\Lambda}^*\pr{x} =\pr{\Lambda (0) +kx}^{+} =0=V(x).
\end{equation*}
 If $x\in \big[ -\frac{V(0)}{k} , 0 \big)$, then $V(x) =V(0) +kx $ and, owing to Eq.  \eqref{ThetaExt2},
\begin{equation*}
\begin{split}
& \sup_{\gamma\in\Theta_t(x)}\set{\int \pr{d\gamma_2-kd\gamma_3}+\int e^{-qs_1}{\Lambda}^*\pr{y_1}d\gamma_0 } \\
=& \sup_{\gamma'\in\Theta_t (0)}\set{\int \pr{d\gamma'_2-kd\gamma'_3}+\int e^{-qs_1}{\Lambda}^*\pr{y_1}d\gamma'_0 } -k\delta_{-x}^{\#}\pr{\mathbb{R}_+} \\
=& V(0) -k\int_0^{-x}dy_1=V(0)+kx.
\end{split}
\end{equation*}

Let us turn to the dual formulation and fix $x<0$.  If $\phi$ is a regular test function satisfying the assumptions in the right-hand member, $\phi(x)\geq \pr{\phi(0)+kx}^+$. It follows that the right-hand member is grater or equal to $\Lambda(x)$.
\begin{enumerate}
\item For $x\in\left[-\frac{V(0)}{k},0\right)$,  we assume, by contradiction, that the inequality is strict. In particular, there exists $\varepsilon>0$ such that, for every $\phi$ satisfying the assumptions in the right-hand member, $\phi(x)\geq V(x)+2\varepsilon$.  We use $\phi_n$ constructed in Proposition \ref{PropSupersol} (for $V$) and  Remark \ref{dPsi} (second assertion).  For $n$ large enough,  $\phi_n'(y)=k$ on $\pp{x,\frac{-2}{n}}$ such that $V(0)+\frac{\tilde{c}}{n}\geq \phi_n(0)\geq \phi_n\pr{-\frac{2}{n}}=\phi_n\pr{x}+k\pr{-x-\frac{2}{n}}\geq V(0)-\frac{2k}{n}+2\varepsilon.$ A contradiction appears when taking $n\rightarrow\infty$.
\item For $x<-\frac{V(0)}{k}$,  it suffices to note that $0\leq \Lambda(x)\leq \phi_n(x)\leq \frac{\tilde{c}}{n}$ (again due to Proposition \ref{PropSupersol}).
\end{enumerate}
\end{proof}

\subsection{The moment of matrices}\label{Sub7.2}
\hspace{1.5em}Let us give some further details on the moment matrices $M_r (\mathbf{m})$ and $M_{r-\lceil \frac{\text{deg} \theta_j}{2} \rceil} \pr{\theta_j \mathbf{m}^y}$ mentioned in Section 5.3.1. We recall that the notation $\lceil x \rceil$ denotes the largest integer not exceeding $x$. In this section, we use $\mathbf{v}^*$ to denote the transpose of vector $\mathbf{v}$.

Given the basis $\pr{1, t, y, t^2 , ty , y^2, t^3,  \cdots, t^r, \cdots, y^r }$ of polynomial $h (t,y)\in \mathbb{R}_r \pp{t,y}$ and $\mathbf{h}$ denotes a vector of the corresponding coefficients $\pr{h_{00} ,h_{10} ,h_{01}, \cdots, h_{0r}}$ of $h(t,y)$. Then, the moment matrix $M_r \pr{\mathbf{m}}$ is defined as follows:
\begin{equation*}
\begin{split}
M_r \pr{\mathbf{m}} &:= \int_{\pp{0,T}\times \mathbf{Y}} \pr{1, t, y, t^2 , ty , y^2, t^3,  \cdots, t^r, \cdots, y^r }^* \pr{1, t, y, t^2 , ty , y^2, t^3,  \cdots, t^r, \cdots, y^r } d\mu  \\
&=\begin{pmatrix}
\mathbf{m}_{00} & \mathbf{m}_{10} &\mathbf{m}_{01}&\cdots & \mathbf{m}_{0r} \\ \mathbf{m}_{10} &  \mathbf{m}_{20} &\mathbf{m}_{11}&\cdots& \mathbf{m}_{1r}\\
  \vdots              &                 \vdots &      \vdots         &  \ddots &   \vdots \\
\mathbf{m}_{0r} &\mathbf{m}_{1r} &\mathbf{m}_{0(r+1)}&\cdots & \mathbf{m}_{0(2r)} \\
\end{pmatrix}, \\
\end{split}
\end{equation*}
and one gets $\displaystyle{L_{\mathbf{m}} \pr{h^2} =\mathbf{h}^T \int_{\pp{0,T}\times \mathbf{Y}} \pr{1, t, y,   \cdots, y^r }^* \pr{1, t, y,\cdots, y^r } d\mu  \ \mathbf{h} = \scal{\mathbf{h} ,M_r \pr{\mathbf{m}} \mathbf{h }} }$.

The polynomial $h (y)$ of degree at most $r-\lceil \frac{\text{deg}\theta_j}{2} \rceil$ has basis $\pr{1, y,\cdots,y^{r-\lceil \frac{\text{deg}\theta_j}{2} \rceil}}$ and the corresponding vector of coefficients $\pr{h_{0} ,h_{1}, \cdots ,h_{r-\lceil \frac{\text{deg}\theta_j}{2} \rceil}}$ is denoted by $\mathbf{h}$. By abuse of notation, we let $\theta_j\mathbf{m}^y$ denote the moments $\bar{\mathbf{m}}^y$ associated with the measure $\bar\mu(dz):=\theta_j(z)\mu^y(dz)$. This leads to $\displaystyle{\theta_j\mathbf{m}^y_i:=\int_{\mathbf{Y}} z^i\theta_j(z)\mu^y(dz)}$. With this notation, we define, as before,\[M_{r-\lceil\frac{\deg\ \theta_j}{2}\rceil}\pr{\theta_j\mathbf{m}^y}:=\int_{\mathbf{Y}} \pr{1,z,z^2,...,z^{r-\lceil\frac{\deg\ \theta_j}{2}\rceil}}^*\pr{1,z,z^2,...,z^{r-\lceil\frac{\deg\ \theta_j}{2}\rceil}}\theta_j(z)\mu(dz).\]
As a consequence, one has $L_{\mathbf{m}^y} \pr{\theta_j h^2 (y)} = \scal{\mathbf{h} ,M_{r-\lceil \frac{\text{deg} \theta_j }{2}\rceil } \pr{ \theta_j\mathbf{m}^y} \mathbf{h }} $.

\bibliographystyle{abbrv}
\bibliography{reference}

\end{document}